\documentclass{amsart}
\usepackage{amssymb,amsmath,amsthm}
\usepackage[dvips]{graphicx}

\let\oldmarginpar\marginpar
\renewcommand\marginpar[1]{\-\oldmarginpar[\raggedleft\footnotesize #1]%
{\raggedright\footnotesize #1}}

\begin{document}

\newtheorem{theorem}{Theorem}[section]
\newtheorem{corollary}[theorem]{Corollary}
\newtheorem{lemma}[theorem]{Lemma}
\newtheorem{proposition}[theorem]{Proposition}
\theoremstyle{definition}
\newtheorem{definition}[theorem]{Definition}
\theoremstyle{remark}
\newtheorem{remark}[theorem]{Remark}
\theoremstyle{definition}
\newtheorem{example}[theorem]{Example}

\numberwithin{equation}{section}

\def\R{{\mathbb R}}
\def\H{{\mathbb H}}
\def\rank{{\text{rank}\,}}
\def\bd{{\partial}}

\title[Pointwise slant and pointwise semi-slant submanifolds]{Pointwise slant and pointwise semi-slant submanifolds in almost contact metric manifolds}

\author{Kwang-Soon Park}
\address{Department of Mathematical Sciences, Seoul National University, Seoul 151-747, Republic of Korea}
\email{parkksn@gmail.com}



\begin{abstract}
As a generalization of slant submanifolds and semi-slant submanifolds, we introduce the notions of pointwise slant submanifolds
and pointwise semi-slant sunmanifolds of an almost contact metric manifold. We obtain a characterization at each notion, investigate
the topological properties of pointwise slant submanifolds, and give some examples of them. We also consider some distributions on cosymplectic,
Sasakian, Kenmotsu manifolds and deal with some properties of warped product pointwise semi-slant submanifolds. Finally, we give some inequalities
for the squared norm of the second fundamental form in terms of a warping function and a semi-slant function for warped product
submanifolds of cosymplectic, Sasakian, Kenmotsu manifolds.
\end{abstract}

\maketitle
\section{Introduction}\label{intro}
\addcontentsline{toc}{section}{Introduction}

Given a Riemannian manifold $(N,g)$ with some additional structures, there are several kinds of submanifolds:

almost complex submanifolds (\cite{K}, \cite{G11}, \cite{K0}, \cite{N}),
totally real submanifolds (\cite{CO}, \cite{E0}, \cite{CDVV}, \cite{CN}),
CR submanifolds (\cite{YK}, \cite{B0}, \cite{C333}, \cite{S0}),
QR submanifolds (\cite{B00}, \cite{KP}, \cite{BF}, \cite{GSK}),
slant submanifolds ((\cite{C3}, \cite{C}, \cite{L}, \cite{CCFF}, \cite{MMT}, \cite{S}),
pointwise slant submanifolds (\cite{E1}, \cite{CG}),
semi-slant submanifolds (\cite{P}, \cite{S00}, \cite{CCFF0}, \cite{KKK}),
pointwise semi-slant submanifolds \cite{S2}, pointwise almost h-slant submanifolds and pointwise almost h-semi-slant submanifolds \cite{P4}, etc.

As a generalization of almost complex submanifolds and totally real submanifolds of an almost Hermitian manifold, B. Y. Chen \cite{C3}
introduced a slant submanifold of an almost Hermitian manifold in 1990. After that, many geometers studied
slant submanifolds (\cite{C}, \cite{L}, \cite{CCFF}, \cite{MMT}, \cite{S}, \cite{ACCM}, etc.).

As a generalization of CR-submanifolds and slant submanifolds of an almost Hermitian manifold, N. Papaghiuc \cite{P} defined the notion
of semi-slant submanifolds of an almost Hermitian manifold in 1994. After that, many geometers investigated
semi-slant submanifolds (\cite{S00}, \cite{CCFF0}, \cite{KKK}, \cite{KK}, \cite{LL}, etc.).

As a generalization of slant submanifolds of an almost Hermitian manifold, F. Etayo \cite{E1} introduced pointwise slant submanifolds
of an almost Hermitian manifold under the name of quasi-slant submanifolds in 1998. After that, B. Y. Chen and O. J. Garay \cite{CG}
studied pointwise slant submanifolds and obtained many nice results in 2012. B. Sahin \cite{S2} also introduced pointwise semi-slant
submanifolds of a K\"{a}hler manifold in 2013. Using this notion, he obtained nice results on wraped product submanifolds of a K\"{a}hler manifold.
Moreover, the author \cite{P4} studied the notions of pointwise almost h-slant submanifolds and pointwise almost h-semi-slant submanifolds of
an almost quaternionic Hermitian manifold in 2014.

As a generalization of slant submanifolds and semi-slant submanifolds of an almost contact metric manifold, we will define the notions of
pointwise slant submanifolds and pointwise semi-slant submanifolds of an almost contact metric manifold.
Throughout the paper, we will see the similarity and the difference among cosymplectic manifolds, Sasakian manifolds, and Kenmotsu manifolds. 

The paper is organized as follows. In section 2 we remind some notions, which are used later.
In section 3 we recall some notions in almost contact metric manifolds, which are also used later.
In section 4 we give the definition of pointwise slant submanifolds of an almost contact metric manifold,
obtain a characterization, and deal with some properties of pointwise slant submanifolds.
In section 5 we investigate the topological properties of pointwise slant submanifolds of a cosymplectic manifold.
In section 6 we give some examples of pointwise slant submanifolds.
In section 7 we introduce the notion of pointwise semi-slant submanifolds of an almost contact metric manifold and obtain a characterization.
In section 8 we consider some distributions on cosymplectic, Sasakian, Kenmotsu manifolds and deal with the notion of totally umbilic submanifolds.
In section 9 we get the non-existence of some type of warped product pointwise semi-slant submanifolds and investigate the properties
of some warped product pointwise semi-slant submanifolds.
In section 10 we obtain inequalities for the squred norm of the second fundamental form in terms of a warping function and
a semi-slant function for a warped product submanifold in cosymplectic, Sasakian, Kenmotsu manifolds.
Finally, we give some examples of pointwise semi-slant submanifolds.

\section{Preliminaries}\label{Prel}

Let $(N, g)$ be a Riemannian manifold, where $N$ is a $n$-dimensional $C^{\infty}$-manifold
and $g$ is a Riemannian metric on $N$.
Let $M$ be a $m$-dimensional submanifold of $(N, g)$.

Denote by $TM^{\perp}$ the normal bundle of $M$ in $N$.

Denote by $\nabla$ and $\overline{\nabla}$ the Levi-Civita connections of $M$ and $N$, respectively.

Then the {\em Gauss} and {\em Weingarten formulas} are given by
\begin{eqnarray}
  \overline{\nabla}_X Y & = & \nabla_X Y + h(X, Y), \label{eq: gauss} \\
  \overline{\nabla}_X Z & = & -A_Z X + D_X Z, \label{eq: weing}
\end{eqnarray}
respectively, for tangent vector fields $X,Y\in \Gamma(TM)$ and a normal vector field $Z\in \Gamma(TM^{\perp})$, where $h$ denotes
the {\em second fundamental
form}, $D$ the {\em normal connection}, and $A$ the {\em shape operator} of $M$ in $N$.

The second fundamental form and the shape operator are related by
\begin{equation}\label{eq: shape}
\langle A_Z X, Y \rangle = \langle h(X, Y), Z \rangle,
\end{equation}
where $\langle \ , \ \rangle$ denotes the induced metric on $M$ as well as the Riemannian metric $g$ on $N$.

Choose a local orthonormal frame
$\{ e_1, \cdots, e_n \}$ of $TN$ such that $e_1, \cdots, e_m$ are tangent to $M$ and $e_{m+1}, \cdots, e_n$ are normal to $M$.

Then the {\em mean curvature vector} $H$ is defined by
\begin{equation}\label{eq: mean}
H := \frac{1}{m} trace \ h = \frac{1}{m}\sum_{i=1}^m h(e_i, e_i)
\end{equation}
and the {\em squared mean curvature} is given by $H^2 := \langle H, H \rangle$.

The {\em squared norm of the second fundamental form} $h$ is defined by
\begin{equation}\label{eq: sqnorm}
|| h ||^2 := \sum_{i,j=1}^m \langle h(e_i, e_j), h(e_i, e_j) \rangle.
\end{equation}
Let $(B, g_B)$ and $({\overline{F}}, g_{\overline{F}})$ be Riemannian manifolds, where $g_B$ and
$g_{\overline{F}}$ are Riemannian metrics on manifolds $B$ and ${\overline{F}}$, respectively.
Let $f$ be a positive $C^{\infty}$-function on $B$.
Consider the product manifold $B\times {\overline{F}}$ with the natural projections $\pi_1 : B\times {\overline{F}} \mapsto B$
 and $\pi_2 : B\times {\overline{F}} \mapsto {\overline{F}}$. The {\em warped product manifold} $M=B\times_f {\overline{F}}$ is
 the product manifold $B\times {\overline{F}}$ equipped with the Riemannian
 metric $g_M$ such that
\begin{equation}\label{eq: wmetr}
|| X ||^2 = || d\pi_1 X ||^2 + f^2(\pi_1(x)) || d\pi_2 X ||^2
\end{equation}
for any tangent vector $X\in T_x M$, $x\in M$.

Hence,
$$
g_M = g_B + f^2g_{\overline{F}}.
$$
We call the function $f$ the {\em warping function} of the warped product
manifold $M$ \cite {C6}.

If the warping function $f$ is constant, then the warped product manifold $M$ is called {\em trivial}.

Given vector fields $X\in \Gamma(TB)$ and $Y\in \Gamma(T{\overline{F}})$, we get their natural horizontal lifts $\widetilde{X},\widetilde{Y}\in \Gamma(TM)$
such that $d\pi_1 \widetilde{X} = X$ and $d\pi_2 \widetilde{Y} = Y$.

For convenience, we will identify $\widetilde{X}$ and $\widetilde{Y}$
with $X$ and $Y$, respectively.

Choose a local orthonormal frame $\{ e_1, \cdots, e_m \}$ of
the tangent bundle $TM$ of $M$ such that $e_1, \cdots, e_{m_1}\in \Gamma(TB)$ and $e_{m_1+1}, \cdots, e_{m}\in \Gamma(T{\overline{F}})$.

Then we have
\begin{equation}\label{eq: lapl2}
\triangle f = \sum_{i=1}^{m_1} ((\nabla_{e_i} e_i)f - e_i^2 f).
\end{equation}
Given unit vector fields $X,Y\in \Gamma(TM)$ such that $X\in \Gamma(TB)$ and $Y\in \Gamma(TF)$, we obtain
\begin{equation}\label{eq: warp2}
\nabla_X Y = \nabla_Y X = (X \ln f)Y,
\end{equation}
where $\nabla$ is the Levi-Civita connection of $(M, g_M)$.

Thus,
\begin{eqnarray}
K(X\wedge Y)
& = & \langle \nabla_Y\nabla_X X - \nabla_X\nabla_Y X, Y \rangle   \label{sect3}    \\
& = & \frac{1}{f} ((\nabla_X X)f - X^2 f),  \nonumber
\end{eqnarray}
where $K(X\wedge Y)$ denotes the sectional curvature of the plane $<X, Y>$ spanned by $X$ and $Y$ over $\mathbb{R}$.

Hence,
\begin{equation}\label{eq: sect4}
\frac{\triangle f}{f} = \sum_{i=1}^{m_1} K(e_i\wedge e_j)
\end{equation}
for each $j = m_1+1, \cdots, m$.

Throughout this paper, we will use the above notations.

\section{Almost contact metric manifolds}\label{contact}

In this section, we remind some notions in almost contact metric manifolds and we will use them later.

Let $N$ be a $(2n+1)$-dimensional $C^{\infty}$-manifold with a tensor field $\phi$ of type $(1,1)$, a vector field $\xi$, and a $1$-form $\eta$ such that
\begin{equation}\label{eq: struc2}
\phi^2 = -I + \eta\otimes \xi, \quad \eta(\xi) = 1,
\end{equation}
where $I$ denotes the identity endomorphism of $TN$.
Then we have \cite {B1}
\begin{equation}\label{eq: struc3}
\phi \xi = 0, \quad \eta \circ \phi = 0.
\end{equation}
We call $(\phi,\xi,\eta)$ an {\em almost contact structure} and $(N,\phi,\xi,\eta)$ an {\em almost contact manifold}.
If there is a Riemannian metric $g$ on $N$ such that
\begin{equation}\label{eq: struc4}
g(\phi X, \phi Y) = g(X, Y) - \eta(X)\eta(Y)
\end{equation}
for any vector fields $X,Y \in \Gamma(TN)$, then we call $(\phi,\xi,\eta,g)$ an {\em almost contact metric structure} and $(N,\phi,\xi,\eta,g)$
an {\em almost contact metric manifold}. The metric $g$ is called a {\em compatible metric}. By replacing $Y$ by $\xi$ at (\ref {eq: struc4}),
we obtain
\begin{equation}\label{eq: struc5}
\eta(X) = g(X, \xi).
\end{equation}
Define $\Phi(X, Y) := g(X, \phi Y)$ for vector fields $X,Y\in \Gamma(TN)$.
Since $\phi$ is anti-symmetric with respect to $g$, the tensor $\Phi$ is a 2-form on $N$ and is called the {\em fundamental 2-form} of
the almost contact metric structure $(\phi,\xi,\eta,g)$.
We can also choose a local orthonormal frame $\{ X_1,\cdots,X_n,\phi X_1,\cdots,\phi X_n,\xi \}$ of $TN$ and we call it a {\em $\phi$-frame}.
An almost contact metric manifold $(N,\phi,\xi,\eta,g)$ is said to be a {\em contact metric manifold} (or {\em almost Sasakian manifold}) \cite {FIP}
if it satisfies
\begin{equation}\label{eq: struc6}
\Phi = d\eta.
\end{equation}
It is easy to check that given a contact metric manifold $(N,\phi,\xi,\eta,g)$, we get
\begin{equation}\label{eq: struc7}
(d\eta)^n\wedge \eta \neq 0.
\end{equation}
The {\em Nijenhuis tensor} of a tensor field $\phi$ is defined by
\begin{equation}\label{eq: struc8}
N(X, Y) := \phi^2 [X,Y] + [\phi X,\phi Y] - \phi [\phi X,Y] - \phi [X,\phi Y]
\end{equation}
for any vector fields $X,Y\in \Gamma(TN)$.
We call the almost contact metric structure  $(\phi,\xi,\eta,g)$ {\em normal} if
\begin{equation}\label{eq: struc9}
N(X, Y) + 2d\eta(X, Y) \xi = 0
\end{equation}
for any vector fields $X,Y\in \Gamma(TN)$.

A contact metric manifold $(N,\phi,\xi,\eta,g)$ is said to be a {\em $K$-contact manifold} if the characteristic vector field $\xi$ is Killing.
It is well-known that for a contact metric manifold $(N,\phi,\xi,\eta,g)$, $\xi$ is Killing if and only if the tensor $\bar{h} := \frac{1}{2} L_{\xi} \phi$
vanishes, where $L$ denotes the Lie derivative \cite {B1}.

Given a contact metric manifold $N = (N,\phi,\xi,\eta,g)$, we know that (i) $\bar{h}$ is a symmetric operator,
(ii) $\overline{\nabla}_X \xi = -\phi X - \phi \bar{h}X$ for $X\in \Gamma(TN)$, where $\overline{\nabla}$ is the Levi-Civita connection of $N$,
(iii) $\bar{h}$ anti-commutes with $\phi$
and $trace (\bar{h}) = 0$ \cite {B1}.
Using the above three properties, A. Lotta proved Theorem \ref {thm: ang02} \cite {L}.

An almost contact metric manifold $(N,\phi,\xi,\eta,g)$ is called a {\em Sasakian manifold} if it is contact and normal.
Given an almost contact metric manifold $(N,\phi,\xi,\eta,g)$, we know that it is Sasakian if and only if
\begin{equation}\label{eq: struc10}
(\overline{\nabla}_X \phi)Y = g(X, Y)\xi - \eta(Y) X
\end{equation}
for $X,Y\in \Gamma(TN)$ \cite {B1}.
If an almost contact metric manifold $(N,\phi,\xi,\eta,g)$ is Sasakian, then we have
\begin{equation}\label{eq: struc11}
\overline{\nabla}_X \xi = -\phi X
\end{equation}
for $X\in \Gamma(TN)$ \cite {B1}.

Moreover, a Sasakian manifold is a $K$-contact manifold \cite {B1}.

An almost contact metric manifold $(N,\phi,\xi,\eta,g)$ is said to be a {\em Kenmotsu manifold} if it satisfies
\begin{equation}\label{eq: struc12}
(\overline{\nabla}_X \phi)Y = g(\phi X, Y)\xi - \eta(Y) \phi X
\end{equation}
for $X,Y\in \Gamma(TN)$ \cite {B1}.
From (\ref {eq: struc12}), by replacing $Y$ by $\xi$, we easily obtain
\begin{equation}\label{eq: struc13}
\overline{\nabla}_X \xi = X - \eta(X)\xi
\end{equation}
for $X\in \Gamma(TN)$ \cite {B1}.

An almost contact metric manifold $(N,\phi,\xi,\eta,g)$ is called an {\em almost cosymplectic manifold} if $\eta$ and $\Phi$ are closed.
An almost cosymplectic manifold $(N,\phi,\xi,\eta,g)$ is said to be a {\em cosymplectic manifold} if it is normal \cite {FIP}.
Given an almost contact metric manifold $(N,\phi,\xi,\eta,g)$, we also know that it is cosymplectic if and only if
$\phi$ is parallel (i.e., $\overline{\nabla} \phi = 0$) \cite {B1}.

Given a cosymplectic manifold $(N,\phi,\xi,\eta,g)$, we easily get
\begin{equation}\label{eq: struc14}
\overline{\nabla} \phi = 0, \ \overline{\nabla} \eta = 0, \ \text{and} \ \overline{\nabla} \xi = 0.
\end{equation}

Throughout this paper, we will use the above notations.

\section{Pointwise slant submanifolds}\label{slant}

In this section we define the notion of pointwise slant submanifolds of an almost contact metric manifold and study its properties.

\begin{definition}
Let $N = (N,\phi,\xi,\eta,g)$ be a $(2n+1)$-dimensional almost contact metric manifold and $M$ a submanifold of $N$.
The submanifold $M$ is called a {\em pointwise slant submanifold} if at each given point $p\in M$ the angle $\theta = \theta(X)$
between $\phi X$ and the space $M_p$ is constant for nonzero $X\in M_p$, where $M_p := \{ X\in T_p M \mid g(X, \xi(p)) = 0 \}$.
\end{definition}

We call the angle $\theta$ a {\em slant function} as a function on $M$.

\begin{remark}
\begin{enumerate}
\item In other papers (\cite {L}, \cite {CCFF}, etc.), the slant angle $\theta$ of a submanifold $M$ in an almost contact metric manifold
$(N,\phi,\xi,\eta,g)$ is defined differently as follows:

Assume that $\xi\in \Gamma(TM)$. Given a point $p\in M$, if the angle $\theta = \theta(X)$ between $\phi X$ and $T_p M$ is constant for nonzero
$X\in T_p M - \{ \xi(p) \}$, then we call the angle $\theta$ a slant angle.

Two definitions for the slant angle of a submanifold in an almost contact metric manifold are essentially same.
But our definition has some advantages as follows: First of all, our definition does not depend on
whether the vector field $\xi$ is tangent to $M$ or not.
Secondly, we have more simple form like this (See Lemma \ref {lem: ang02}): $T^2 X = -\cos^2 \theta X$
for $X\in M_p$, which is the same form with the case of an almost Hermitian manifold, etc..

\item If $\theta : M \mapsto [0, \frac{\pi}{2})$, then by using Theorem 3.3 of \cite {L}, we obtain that either $\xi$ is tangent to $M$
or $\xi$ is normal to $M$.

\item Like Examples of section \ref{exam1}, we need to deal with our notion
both when $\xi$ is tangent to $M$ and when $\xi$ is normal to $M$ so that by (1), our definition is more favorite.
\end{enumerate}
\end{remark}

\begin{remark}
\begin{enumerate}
\item If the slant function $\theta$ is constant on $M$, then we call $M$ a {\em slant submanifold}.

\item If $\theta = 0$ on $M$, (which implies $\phi(TM) \subset TM$), then we call $M$ an {\em invariant submanifold}.

\item If $\theta = \frac{\pi}{2}$ on $M$, (which implies $\phi(TM) \subset TM^{\perp}$), then we call $M$ an {\em anti-invariant submanifold}.
\end{enumerate}
\end{remark}

Let $M$ be a pointwise slant submanifold of an almost contact metric manifold $(N,\phi,\xi,\eta,g)$ with the slant function $\theta$.

For $X\in \Gamma(TM)$, we write
\begin{equation}\label{eq: tens1}
\phi X = TX + FX,
\end{equation}
where $TX\in \Gamma(TM)$ and $FX\in \Gamma(TM^{\perp})$.

For $Z\in \Gamma(TM^{\perp})$, we obtain
\begin{equation}\label{eq: tens2}
\phi Z = tZ + fZ,
\end{equation}
where $tZ\in \Gamma(TM)$ and $fZ\in \Gamma(TM^{\perp})$.

Let $\displaystyle{T^1 M := \bigcup_{p\in M} M_p = \bigcup_{p\in M} \{ X\in T_p M \mid g(X, \xi(p)) = 0 \}}$.

Then we get

\begin{lemma}\label{lem: ang02}
Let $M$ be a submanifold of an almost contact metric manifold $N = (N,\phi,\xi,\eta,g)$.
Then $M$ is a pointwise slant submanifold of $N$ if and only if $T^2 = -\cos^2 \theta\cdot I$ on $T^1 M$ for some function $\theta : M\mapsto \mathbb{R}$.
\end{lemma}

\begin{proof}
Suppose that $M$ is a pointwise slant submanifold of $N$ with the slant function $\theta : M\mapsto \mathbb{R}$. Given a point $p\in M$,
if $\theta(p) = \frac{\pi}{2}$, then trivial! If $\theta(p) \neq \frac{\pi}{2}$, then for any nonzero $X\in M_p$ we have
\begin{equation}\label{eq: ang02}
\cos \theta(p) = \frac{g(\phi X, TX)}{||\phi X|| \ ||TX||} = \frac{||TX||}{||X||}
\end{equation}
so that $\cos^2 \theta(p) g(X, X) = g(TX, TX) = -g(T^2X, X)$.
Replacing $X$ by $X+Y$, $Y\in M_p$, we obtain
\begin{equation*}\label{eq: ang03}
g((T^2 + \cos^2 \theta(p) I)X, Y) + g(X, (T^2 + \cos^2 \theta(p) I)Y) = 0.
\end{equation*}
But $T^2 + \cos^2 \theta(p) I$ is also symmetric so that
\begin{equation}\label{eq: ang04}
(T^2 + \cos^2 \theta(p) I)X = 0.
\end{equation}
Conversely, if $T^2 = -\cos^2 \theta I$ on $T^1 M$ for some function $\theta : M\mapsto \mathbb{R}$, then we have
$g(TX, TX) = -g(T^2 X, X) = \cos^2 \theta(p) g(X, X)$ for any nonzero $X\in M_p$, $p\in M$
so that
\begin{equation}\label{eq: ang05}
\cos^2 \theta(p) = \frac{g(TX, TX)}{g(X, X)},
\end{equation}
which implies that $\arccos(|\cos \theta(p)|)$ is a slant function on $M$.

Hence, $M$ is a pointwise slant submanifold of $N$.
\end{proof}

\begin{remark}
Let $M$ be a pointwise slant submanifold of an almost contact metric manifold $(N,\phi,\xi,\eta,g)$ with the slant function $\theta$.
By using Lemma \ref {lem: ang02}, we easily get
\begin{equation}\label{eq: ang06}
g(TX, TY) = \cos^2 \theta g(X, Y),
\end{equation}
\begin{equation}\label{eq: ang07}
g(FX, FY) = \sin^2 \theta g(X, Y),
\end{equation}
for $X,Y\in \Gamma(T^1 M)$.
At each given point $p\in M$ with $0\leq \theta(p)< \frac{\pi}{2}$, by using (\ref {eq: ang06}) we can choose an orthonormal basis
$\{ X_1,\sec \theta TX_1, \cdots, X_k,\sec \theta TX_k \}$ of $M_p$.
\end{remark}

Using Lemma \ref {lem: ang02}, we obtain

\begin{corollary}
Let $M$ be a pointwise slant submanifold of an almost contact metric manifold $(N,\phi,\xi,\eta,g)$ with the nonconstant slant function
$\theta : M\mapsto \mathbb{R}$.
Then $M$ is even-dimensional.
\end{corollary}

In a similar way to Proposition 2.1 of \cite {E1}, we have

\begin{proposition}\label{prop: ang02}
Let $M$ be a $2$-dimensional submanifold of an almost contact metric manifold $(N,\phi,\xi,\eta,g)$.
Then $M$ is a pointwise slant submanifold of $N$.
\end{proposition}

\begin{proof}
Given a point $p\in M$, we consider it at two cases.

If $\xi\notin \Gamma(T_p M^{\perp})$, then since $\dim M_p = 1$ and $g(\phi X, X) = 0$ for $X\in M_p$, we immediately obtain $\theta(p) = \frac{\pi}{2}$.

If $\xi\in \Gamma(T_p M^{\perp})$, then we choose an orthonormal basis $\{ X,Y \}$ of $T_p M$. Let $\alpha := g(X, \phi Y)$.
Given any nonzero vector $Z=aX + bY\in T_p M$, $a,b\in \mathbb{R}$, we get
\begin{equation*}\label{eq: ang08}
TZ = g(\phi Z, X)X + g(\phi Z, Y)Y = bg(X, \phi Y)X - ag(X, \phi Y)Y = b\alpha X - a\alpha Y
\end{equation*}
so that
\begin{equation*}\label{eq: ang09}
\cos \theta(Z) = \frac{g(\phi Z, TZ)}{||\phi Z|| \ ||TZ||} = \frac{||TZ||}{||Z||} = |\alpha|,
\end{equation*}
which means the result.
\end{proof}

\begin{remark}
Proposition \ref {prop: ang02} gives us a kind of examples for pointwise slant submanifolds.
\end{remark}

In a similar way to Theorem 2.4 of \cite {E1}, we obtain

\begin{theorem}\label{thm: ang01}
Let $M$ be a pointwise slant connected totally geodesic submanifold of a cosymplectic manifold $(N,\phi,\xi,\eta,g)$.
Then $M$ is a slant submanifold of $N$.
\end{theorem}

\begin{proof}
Given any two points $p,q\in M$, we choose a $C^{\infty}$-curve $c: [0,1]\mapsto M$ such that $c(0)=p$ and $c(1)=q$. For nonzero $X\in M_p$, we take
a parallel transport $Z(t)$ along the curve $c$ in $M$ such that $Z(0)=X$ and $Z(1)=Y$. Then since $M$ is totally geodesic,
\begin{equation}\label{eq: ang010}
0 = \nabla_{c'} Z(t) = \overline{\nabla}_{c'} Z(t),
\end{equation}
where $\nabla$ and $\overline{\nabla}$ are the Levi-Civita connections of $M$ and $N$, respectively.
By the uniqueness of parallel transports, $Z(t)$ is also a parallel transport in $N$.
Since $\xi$ is parallel (See (\ref {eq: struc14})), we have
\begin{equation}\label{eq: ang011}
\frac{d}{dt} g(Z(t), \xi) = g(\overline{\nabla}_{c'} Z(t), \xi) + g(Z(t), \overline{\nabla}_{c'} \xi) = 0 \quad \text{and} \ g(Z(0), \xi) = 0
\end{equation}
so that
\begin{equation*}\label{eq: ang012}
0 = g(Z(1), \xi) =  g(Y, \xi),
\end{equation*}
which implies $Y\in M_q$.

But by (\ref {eq: struc14}),
\begin{equation*}\label{eq: ang013}
\overline{\nabla}_{c'} \phi Z(t) = (\overline{\nabla}_{c'} \phi) Z(t) + \phi \overline{\nabla}_{c'} Z(t) = 0
\end{equation*}
so that $\phi Z(t)$ becomes a parallel transport along $c$ in $N$ such that $\phi Z(0) = \phi X$ and $\phi Z(1) = \phi Y$.

Define a map $\tau : T_p N \mapsto T_q N$ by $\tau(U) = V$ for $U\in T_p N$ and $V\in T_q N$, where $W(t)$ is the parallel transport along $c$ in $N$
such that $W(0)=U$ and $W(1)=V$.
Then $\tau$ is surely isometry. It is easy to check that $\tau(T_p M) = T_p M$ and $\tau(T_p M^{\perp}) = T_p M^{\perp}$ so that
$\tau(\phi X) = \phi Y$ means $\tau(TX) = TY$.

Hence,
\begin{equation*}\label{eq: ang014}
\cos \theta(p) = \frac{||TX||}{||X||} = \frac{||TY||}{||Y||} = \cos \theta(q),
\end{equation*}
where $\theta$ is the slant function on $M$.

Therefore, the result follows.
\end{proof}

Using Proposition \ref {prop: ang02} and Theorem \ref {thm: ang01}, we get

\begin{corollary}\label{cor: ang01}
Let $M$ be a $2$-dimensional connected totally geodesic submanifold of a cosymplectic manifold $(N,\phi,\xi,\eta,g)$.
Then $M$ is a slant submanifold of $N$.
\end{corollary}

\begin{remark}
Corollary \ref {cor: ang01} gives us a kind of examples for slant submanifolds.
\end{remark}

Now, we need to mention A. Lotta's result \cite {L}, which is the generalization of the well-known result of K. Yano and M. Kon \cite {YK0}.

\begin{theorem}\label{thm: ang02} \cite {L}
Let $M$ be a submanifold of a contact metric manifold $N = (N,\phi,\xi,\eta,g)$.
If $\xi$ is normal to $M$, then $M$ is a anti-invariant submanifold of $N$.
\end{theorem}

\begin{remark}
\begin{enumerate}
\item As we know, Theorem \ref {thm: ang02} is very strong and it implies that there do not exist submanifolds $M$ with $\xi\in \Gamma(TM^{\perp})$
in a contact metric manifold $(N,\phi,\xi,\eta,g)$ such that either $\{ X,\phi X \} \subset M_p$ for some nonzero $X\in M_p$, $p\in M$
or $2\dim M > \dim N + 1$.

\item  If $N$ is either cosymplectic or Kenmotsu,
then Theorem \ref {thm: ang02} is not true (See Example \ref {ex sl01} and Example \ref {ex sl02}) and
we easily check that the argument of the proof of Theorem \ref {thm: ang02} at \cite {L} does not give any information anymore.

\item In the view point of (1) and (2), we may think that Sasakian manifolds are somewhat different from cosymplectic manifolds
and Kenmotsu manifolds (See section 8, section 9, section 10).
\end{enumerate}
\end{remark}

In the same way to Proposition 2.1 of \cite {CG}, we can obtain

\begin{proposition}
Let $M$ be a submanifold of an almost contact metric manifold $(N,\phi,\xi,\eta,g)$.
Then $M$ is a pointwise slant submanifold of $N$ if and only if
\begin{equation}\label{eq: prop01}
g(TX, TY) = 0 \ \text{whenever} \ g(X, Y) = 0 \ \text{for} \ X,Y\in M_p, p\in M.
\end{equation}
\end{proposition}

Considering slant functions as conformal invariant, we easily derive

\begin{proposition}
Let $M$ be a pointwise slant submanifold of an almost contact metric manifold $(N,\phi,\xi,\eta,g)$ with the slant function
$\theta : M\mapsto \mathbb{R}$.
Then for any given $C^{\infty}$-function $f : N\mapsto \mathbb{R}$, $M$ is also a pointwise slant submanifold of
an almost contact metric manifold $(N,\phi,e^{-f}\xi,e^f\eta,e^{2f}g)$ with the same slant function $\theta$.
\end{proposition}

\begin{theorem}
Let $M$ be a slant submanifold of an almost contact metric manifold $N = (N,\phi,\xi,\eta,g)$  with the slant angle $\theta$.
Assume that $N$ is one of the following three manifolds: cosymplectic, Sasakian, Kenmotsu. Then we have
\begin{equation}\label{eq: rel01}
A_{FX} TX = A_{FTX} X \quad \text{for} \ X\in \Gamma(T^1 M).
\end{equation}
\end{theorem}

\begin{proof}
We will only give its proof when $N$ is Sasakian. For the other cases, we can show them in a similar way.
If $\theta = \frac{\pi}{2}$, then done!
Assume that $0\leq \theta < \frac{\pi}{2}$.
Given a unit vector field $X\in \Gamma(T^1 M)$, we have
\begin{equation}\label{eq: rel02}
TX = \cos \theta \cdot X^*
\end{equation}
for some unit vector field $X^*\in \Gamma(T^1 M)$ with $g(X, X^*) = 0$.
Then for any $Y\in \Gamma(TM)$, by using (\ref {eq: gauss}), (\ref {eq: weing}), and (\ref {eq: struc10}), we obtain
\begin{eqnarray}
\overline{\nabla}_Y (\phi X)
& = & \overline{\nabla}_Y (\cos \theta \cdot X^*) + \overline{\nabla}_Y FX    \label{eq: rel03}    \\
& = &  \cos \theta \cdot \nabla_Y X^* + \cos \theta h(Y, X^*) - A_{FX} Y + D_Y FX   \nonumber
\end{eqnarray}
and
\begin{eqnarray}
\overline{\nabla}_Y (\phi X)
& = &  (\overline{\nabla}_Y \phi) X + \phi\overline{\nabla}_Y X      \label{eq: rel04}           \\
& = & g(Y, X) \xi - \eta(X) Y + T\nabla_Y X + F\nabla_Y X       \nonumber   \\
&   & + th(Y, X) + fh(Y, X)      \nonumber   \\
& = &  g(Y, X) \xi + T\nabla_Y X + F\nabla_Y X + th(Y, X) + fh(Y, X).    \nonumber
\end{eqnarray}
Thus, by taking the inner product of right hand sides of (\ref {eq: rel03}) and (\ref {eq: rel04}) with $X^*$, we derive
\begin{equation*}\label{eq: rel05}
g(-A_{FX} Y, X^*) = g(th(Y, X), X^*),
\end{equation*}
which gives
\begin{equation*}\label{eq: rel05}
g(A_{FX} X^*, Y) = g(A_{FX^*} X, Y).
\end{equation*}
Therefore, the result follows.
\end{proof}

\section{Topological properties}\label{topol}

In this section we investigate the topological properties of pointwise slant submanifolds of a cosymplectic manifold.
A pointwise slant submanifold $M$ of an almost contact metric manifold $(N,\phi,\xi,\eta,g)$ is said to be {\em proper} if the slant function $\theta$
of $M$ in $N$ is given by $\theta : M \mapsto [0,\frac{\pi}{2})$.

Let $M$ be a pointwise slant submanifold of an almost contact metric manifold $(N,\phi,\xi,\eta,g)$.
Given $X,Y\in \Gamma(TM)$, we define
\begin{eqnarray}
&  & (\nabla_X T)Y := \nabla_X (TY) - T\nabla_X Y,   \label{eq: top01}    \\
&  &  (D_X F)Y := D_X (FY) - F\nabla_X Y.    \label{eq: top02}
\end{eqnarray}
We call the tensors $T$ and $F$ {\em parallel} if $\nabla T = 0$ and $\nabla F = 0$, respectively.
Then in a similar way to Lemma 3.8 of \cite {P4}, we easily obtain

\begin{lemma}\label{lem: top01}
Let $M$ be a pointwise slant submanifold of a cosymplectic manifold $(N,\phi,\xi,\eta,g)$.
Then we get
\begin{enumerate}
\item
\begin{eqnarray}
&  &(\nabla_X T) Y = A_{FY} X + th(X,Y),  \label{eq: top03}   \\
&  &(D_X F) Y = -h(X,TY) + fh(X,Y)   \label{eq: top04}
\end{eqnarray}
for $X,Y\in \Gamma(TM)$.

\item
\begin{eqnarray}
&  &-TA_Z X + tD_X Z = \nabla_X (tZ) - A_{fZ} X,  \label{eq: top05}  \\
&  &-FA_Z X + fD_X Z = h(X,tZ) + D_X (fZ)   \label{eq: top06}
\end{eqnarray}
for $X\in \Gamma(TM)$ and $Z\in \Gamma(TM^{\perp})$.
\end{enumerate}
\end{lemma}

Let $M$ be a proper pointwise slant submanifold of a cosymplectic manifold $(N,\phi,\xi,\eta,g)$.

Define
\begin{equation}\label{eq: top07}
\Omega(X, Y) := g(X, TY) \quad \text{for} \ X,Y\in \Gamma(TM).
\end{equation}
Then $\Omega$ is a $2$-form on $M$, which is non-degenerate on $T^1 M$ (\cite {CG}, \cite {P4}).

\begin{theorem}\label{thm: top01}
Let $M$ be a proper pointwise slant submanifold of a cosymplectic manifold $(N,\phi,\xi,\eta,g)$.
Then the 2-form $\Omega$ is closed.
\end{theorem}

\begin{proof}
Given $X,Y,Z\in \Gamma(TM)$, we get
\begin{align*}
3d\Omega(X,Y,Z) &= X\Omega (Y,Z) - Y\Omega (X,Z) + Z\Omega (X,Y)      \\
                  &- \Omega ([X,Y],Z) + \Omega ([X,Z],Y) - \Omega ([Y,Z],X)
\end{align*}
so that
\begin{align*}
3d\Omega (X,Y,Z) &= g(\nabla_X Y,TZ) + g(Y,\nabla_X TZ) - g(\nabla_Y X,TZ)      \\
                  &- g(X,\nabla_Y TZ) + g(\nabla_Z X,TY) + g(X,\nabla_Z TY)  \\
                  &- g(\nabla_X Y - \nabla_Y X,TZ) + g(\nabla_X Z - \nabla_Z X,TY) - g(\nabla_Y Z - \nabla_Z Y,TX)      \\
                  &= g(Y,(\nabla_X T)Z) - g(X,(\nabla_Y T)Z) + g(X,(\nabla_Z T)Y).      \\
\end{align*}
Using Lemma \ref {lem: top01} and (\ref {eq: shape}), we obtain
\begin{align*}
3d\Omega (X,Y,Z) &= g(Y, A_{FZ} X + th(X, Z)) - g(X, A_{FZ} Y + th(Y, Z))     \\
                  &+ g(X, A_{FY} Z + th(Z, Y))  \\
                  &= g(Y,th(X,Z)) - g(Z, th(Y,X))      \\
                  &- g(X,th(Y,Z)) + g(Z, th(X,Y))      \\
                  &+ g(X,th(Z,Y)) - g(Y, th(X,Z))      \\
                  &= 0.
\end{align*}
Therefore, the result follows.
\end{proof}

Consider the restriction of the $1$-form $\eta$ to $M$. We also denote it by $\eta$.

Denote by $[\Omega]$ and $[\eta]$ the de Rham cohomology classes of $2$-form $\Omega$ and $1$-form $\eta$ on $M$, respectively.
As we know, a cosymplectic manifold is locally a Riemannian product of a K\"{a}hler manifold and an interval and the cosymplectic condition
(i.e., $\overline{\nabla}\phi = 0$) naturally corresponds to the K\"{a}hler condition ($\overline{\nabla}J = 0$) (See \cite {FIP}).

Hence, in a similar way to Theorem 5.1 of \cite {CG} and to Theorem 5.2 of \cite {P4}, by using Theorem \ref {thm: top01}, we obtain

\begin{theorem}\label{thm: top02}
Let $M$ be a $2m$-dimensional compact proper pointwise slant submanifold of a $(2n+1)$-dimensional cosymplectic manifold $(N,\phi,\xi,\eta,g)$
such that $\xi$ is normal to $M$.

Then $[\Omega]\in H^2(M, \mathbb{R})$ is non-vanishing.
\end{theorem}

\begin{proof}
Since $TM = T^1 M$, by the definition of $\Omega$, $\Omega$ is non-degenerate on $M$.

Therefore, the result follows.
\end{proof}

\begin{remark}\label{rem: top01}
By the proof of Theorem \ref {thm: top02}, we have
\begin{equation}\label{eq: top08}
\dim H^{2i}(M, \mathbb{R}) \geq 1 \quad \text{for} \ 0\leq i \leq m.
\end{equation}
\end{remark}

\begin{theorem}\label{thm: top03}
Let $M$ be a $(2m+1)$-dimensional compact proper pointwise slant submanifold of a $(2n+1)$-dimensional cosymplectic manifold $(N,\phi,\xi,\eta,g)$
such that $\xi$ is tangent to $M$.

Then both $[\eta]\in H^1(M, \mathbb{R})$ and $[\Omega]\in H^2(M, \mathbb{R})$ are non-vanishing.
\end{theorem}

\begin{proof}
Using (\ref {eq: ang06}), we can choose a local orthonormal frame $\{ \xi,X_1,\sec \theta TX_1, \cdots,$ $X_m,\sec \theta TX_m \}$ of $TM$.

Thus,
\begin{equation}\label{eq: top09}
\eta \wedge \Omega^m = \eta \wedge g(\ , T)^m \neq 0 \ \text{at each point of} \ M
\end{equation}
so that it gives a volume form on $M$.

Hence, both $[\eta]$ and $[\Omega]$ are never vanishing.
\end{proof}

\begin{remark}\label{rem: top02}
By the proof of Theorem \ref {thm: top03}, we get
\begin{equation}\label{eq: top010}
\dim H^{i}(M, \mathbb{R}) \geq 1 \quad \text{for} \ 0\leq i \leq 2m+1.
\end{equation}
\end{remark}

By using (\ref {eq: top08}) and (\ref {eq: top010}), we obtain

\begin{corollary}
Every $m$-sphere $S^m$, $m\geq 3$, cannot be immersed in a cosymplectic manifold as a proper pointwise slant submanifold.
\end{corollary}

\begin{corollary}
Any $m$-dimensional real projective space $\mathbb{R}\mathbb{P}^m$, $m\geq 3$, cannot be immersed in a cosymplectic manifold
as a proper pointwise slant submanifold.
\end{corollary}

\begin{remark}
For $2$-sphere $S^2$ and $2$-torus $T^2$, they satisfy the condition (\ref {eq: top08}). And by Proposition \ref {prop: ang02}, they are pointwise slant
submanifolds of a cosymplectic manifold $(N,\phi,\xi,\eta,g)$ if they are just submanifolds of $N$.
\end{remark}

\section{Examples}\label{exam1}

In this section we give some examples of pointwise slant submanifolds.

\begin{example}
Define a map $i : \mathbb{R}^{3} \mapsto \mathbb{R}^{5}$ by
$$
i(x_1,x_2,x_3) = (y_1,y_2,y_3,y_4,t) = (x_1,\sin x_2,0,\cos x_2,x_3).
$$
Let $M := \{ (x_1,x_2,x_3)\in \mathbb{R}^{3} \mid 0< x_2 <\frac{\pi}{2} \}$.

We define $(\phi,\xi,\eta,g)$ on $\mathbb{R}^{5}$ as follows:
\begin{align*}
&\phi(a_1\tfrac{\partial}{\partial y_1} + \cdots + a_4\tfrac{\partial}{\partial y_4} + a_5\tfrac{\partial}{\partial t}) =
-a_2\tfrac{\partial}{\partial y_1} + a_1\tfrac{\partial}{\partial y_2} - a_4\tfrac{\partial}{\partial y_3} + a_3\tfrac{\partial}{\partial y_4},   \\
&\xi := \tfrac{\partial}{\partial t}, \quad \eta := dt, \ a_i\in \mathbb{R}, 1\leq i\leq 5,
\end{align*}
$g$ is the Euclidean metric on $\mathbb{R}^{5}$.
It is easy to check that $(\phi,\xi,\eta,g)$ is an almost contact metric structure on $\mathbb{R}^{5}$.

Then $M$ is a pointwise slant submanifold of an almost contact metric manifold $(\mathbb{R}^{5},\phi,\xi,\eta,g)$
with the slant function $k(x_1,x_2,x_3) = x_2$ such that $\xi$ is tangent to $M$.
\end{example}

\begin{example}\label{ex sl01}
Define a map $i : \mathbb{R}^{2} \mapsto \mathbb{R}^{5}$ by
$$
i(x_1,x_2) = (y_1,y_2,y_3,y_4,t) = (0,\cos x_1,x_2,\sin x_1,0).
$$
Let $M := \{ (x_1,x_2)\in \mathbb{R}^{2} \mid 0< x_1 <\frac{\pi}{2} \}$.

We define $(\phi,\xi,\eta,g)$ on $\mathbb{R}^{5}$ as follows:
\begin{align*}
&\phi(a_1\tfrac{\partial}{\partial y_1} + \cdots + a_4\tfrac{\partial}{\partial y_4} + a_5\tfrac{\partial}{\partial t}) =
-a_2\tfrac{\partial}{\partial y_1} + a_1\tfrac{\partial}{\partial y_2} - a_4\tfrac{\partial}{\partial y_3} + a_3\tfrac{\partial}{\partial y_4},   \\
&\xi := \tfrac{\partial}{\partial t}, \quad \eta := dt, \ a_i\in \mathbb{R}, 1\leq i\leq 5,
\end{align*}
$g$ is the Euclidean metric on $\mathbb{R}^{5}$.
It is easy to check that $(\phi,\xi,\eta,g)$ is an almost contact metric structure on $\mathbb{R}^{5}$.

We also know that $(\mathbb{R}^{5},\phi,\xi,\eta,g)$ is a cosymplectic manifold.

Then $M$ is a pointwise slant submanifold of a cosymplectic manifold $(\mathbb{R}^{5},\phi,\xi,\eta,g)$
with the slant function $k(x_1,x_2) = x_1$ such that $\xi$ is normal to $M$.
\end{example}

\begin{example}\label{ex sl02}
Let $t$ be a coordinate of $\mathbb{R}$ and $(y_1,y_2,y_3,y_4)$ coordinates of $\mathbb{R}^4$.
Let $N := \mathbb{R} \times_f \mathbb{R}^4$ be a warped product manifold of the Euclidean space $\mathbb{R}$ and the Euclidean space $\mathbb{R}^4$
with the natural projections $\pi_1 : N \mapsto \mathbb{R}$ and $\pi_2 : N \mapsto \mathbb{R}^4$ such that the warping function $f(t) = e^t$.

Let $\mathbb{R}^4 = (\mathbb{R}^4,\bar{g},J)$, where $\bar{g}$ is the Euclidean metric on $\mathbb{R}^4$ and $J$ is a complex structure on $\mathbb{R}^4$
defined by
$$
J(a_1\tfrac{\partial}{\partial y_1} + \cdots + a_4\tfrac{\partial}{\partial y_4}) =
-a_2\tfrac{\partial}{\partial y_1} + a_1\tfrac{\partial}{\partial y_2} - a_4\tfrac{\partial}{\partial y_3} + a_3\tfrac{\partial}{\partial y_4}.
$$
Then $\mathbb{R}^4$ is obviously K\"{a}hler.

We define $(\phi,\xi,\eta,g)$ on $N$ as follows:
\begin{align*}
&\phi(a_1\tfrac{\partial}{\partial y_1} + \cdots + a_4\tfrac{\partial}{\partial y_4} + a_5\tfrac{d}{dt}) :=
J(a_1\tfrac{\partial}{\partial y_1} + \cdots + a_4\tfrac{\partial}{\partial y_4}),   \\
&\xi := \tfrac{d}{dt}, \quad \eta := dt,  \\
&g(Z, W) := \eta(Z) \eta(W) + f(t)^2 \bar{g}(d\pi_2(Z), d\pi_2(W))
\end{align*}
for $Z,W\in \Gamma(TN)$, $a_i\in \mathbb{R}$, $1\leq i\leq 5$.

We easily check that $(\phi,\xi,\eta,g)$ is an almost contact metric structure on $N$.

Furthermore, by Proposition 3 of \cite{K00}, $(N,\phi,\xi,\eta,g)$ is a Kenmotsu manifold.

Let $M := \{ (x_1,x_2)\in \mathbb{R}^{2} \mid 0< x_1 <\frac{\pi}{2} \}$.

Define a map $i : \mathbb{R}^{2} \mapsto \mathbb{R}^{4}\subset N$ by
$$
i(x_1,x_2) = (y_1,y_2,y_3,y_4) = (x_2,\sin x_1,1972,\cos x_1).
$$
Then $M$ is a pointwise slant submanifold of a Kenmotsu manifold $(N,\phi,\xi,\eta,g)$
with the slant function $k(x_1,x_2) = x_1$ such that $\xi$ is normal to $M$.
\end{example}

\begin{example}\label{ex 01}
Let $M$ be a submanifold of a hyperk\"{a}hler manifold $(\overline{M},J_1,J_2,J_3,\bar{g})$ such that $M$ is complex
with respect to the complex structure $J_1$ (i.e., $J_1(TM) = TM$) and totally real with respect to the complex structure
$J_2$ (i.e., $J_2(TM) \subset TM^{\perp}$) \cite {B}.
Let $f : \overline{M} \mapsto [0, \frac{\pi}{2}]$ be a $C^{\infty}$-function and $N := \overline{M}\times \mathbb{R}$ with the natural projections
$\pi_1 : N \mapsto \overline{M}$ and $\pi_2 : N \mapsto \mathbb{R}$.

We define $(\phi,\xi,\eta,g)$ on $N$ as follows:
\begin{align*}
&\phi(X + h\tfrac{d}{dt}) := \cos (f\circ \pi_1) J_1 X - \sin (f\circ \pi_1) J_2 X,   \\
&\xi := \tfrac{d}{dt}, \quad \eta := dt,  \\
&g(Z, W) := \bar{g}(d\pi_1 Z, d\pi_1 W) + \eta(Z)\cdot \eta(W)
\end{align*}
for $X\in \Gamma(T\overline{M})$, $h\in C^{\infty} (N)$, $Z,W\in \Gamma(TN)$, and $t$ is a coordinate of $\mathbb{R}$.

It is easy to show that $(\phi,\xi,\eta,g)$ is an almost contact metric structure on $N$.

Then $M$ is a pointwise slant submanifold of an almost contact metric manifold $(N,\phi,\xi,\eta,g)$
with the slant function $f\circ \pi_1$ such that $\xi$ is normal to $M$.
\end{example}

\begin{example}\label{ex 02}
Given an Euclidean space $\mathbb{R}^{5} = \mathbb{R}^{4} \times \mathbb{R}$ with coordinates $(y_1,\cdots,y_4,t)$,
we consider complex structures $J_1$ and $J_2$ on $\mathbb{R}^{4}$ as follows:
\begin{align*}
  &J_1(\tfrac{\partial}{\partial y_1})=\tfrac{\partial}{\partial y_2},
  J_1(\tfrac{\partial}{\partial y_2})=-\tfrac{\partial}{\partial y_1},
  J_1(\tfrac{\partial}{\partial y_3})=\tfrac{\partial}{\partial y_4},
  J_1(\tfrac{\partial}{\partial y_4})=-\tfrac{\partial}{\partial y_3},     \\
  &J_2(\tfrac{\partial}{\partial y_1})=\tfrac{\partial}{\partial y_3},
  J_2(\tfrac{\partial}{\partial y_2})=-\tfrac{\partial}{\partial y_4},
  J_2(\tfrac{\partial}{\partial y_3})=-\tfrac{\partial}{\partial y_1},
  J_2(\tfrac{\partial}{\partial y_4})=\tfrac{\partial}{\partial y_2},    \\
\end{align*}
Let $f : \mathbb{R}^{5} \mapsto [0, \frac{\pi}{2}]$ be a $C^{\infty}$-function.

We define $(\phi,\xi,\eta,g)$ on $\mathbb{R}^{5}$ as follows:
\begin{align*}
&\phi(X + h\tfrac{d}{dt}) := \cos f \cdot J_1 X - \sin f \cdot J_2 X,   \\
&\xi := \tfrac{d}{dt}, \quad \eta := dt,
\end{align*}
$g$ is the Euclidean metric on $\mathbb{R}^{5}$
for $X\in \Gamma(T\mathbb{R}^{4})$ and $h\in C^{\infty} (\mathbb{R}^{5})$.

We can easily check that $(\phi,\xi,\eta,g)$ is an almost contact metric structure on $\mathbb{R}^{5}$.

Define a map $i : \mathbb{R}^{2} \mapsto \mathbb{R}^{5}$ by
$$
i(x_1,x_2) = (y_1,y_2,y_3,y_4,t) = (e,-\pi,x_2,x_1,\sqrt{2}).
$$
Then $\mathbb{R}^{2}$ is a pointwise slant submanifold of an almost contact metric manifold $(\mathbb{R}^{5},\phi,\xi,\eta,g)$
with the slant function $f$ such that $\xi$ is normal to $\mathbb{R}^{2}$.
\end{example}

\begin{example}\label{ex 03}
With all the conditions of Example \ref {ex 02}, define a function $f : \mathbb{R}^{5} \mapsto [0, \frac{\pi}{2}]$ by
$f(y_1,\cdots,y_4,t) = \arctan (|y_1+y_2+y_3+y_4|)$.

Then $\mathbb{R}^{2}$ is a pointwise slant submanifold of an almost contact metric manifold $(\mathbb{R}^{5},\phi,\xi,\eta,g)$
with the slant function $(f\circ i)(x_1,x_2) = \arctan (|e-\pi+x_1+x_2|)$ such that $\xi$ is normal to $\mathbb{R}^{2}$.
\end{example}

\section{Pointwise semi-slant submanifolds}\label{semi-slant}

In this section we introduce the notion of pointwise semi-slant submanifolds of an almost contact metric manifold and
obtain a characterization of pointwise semi-slant submanifolds.

\begin{definition}\label{def: sslant01}
Let $(N,\phi,\xi,\eta,g)$ be an almost contact metric manifold and $M$ a submanifold of $N$. The submanifold $M$ is called a
{\em pointwise semi-slant submanifold} if there is a distribution $\mathcal{D}_1 \subset TM$ on $M$ such that
$$
TM = \mathcal{D}_1 \oplus \mathcal{D}_2, \quad \phi(\mathcal{D}_1) \subset \mathcal{D}_1,
$$
and at each given point $p\in M$ the angle $\theta = \theta(X)$ between $\phi X$ and the space $(\mathcal{D}_2)_p$ is constant for nonzero
$X\in (\mathcal{D}_2)_p$, where $\mathcal{D}_2$ is the orthogonal complement of $\mathcal{D}_1$ in $TM$.
\end{definition}

We call the angle $\theta$ a {\em semi-slant function} as a function on $M$.

\begin{remark}\label{rem: sslant01}
Let $M$ be a pointwise semi-slant submanifold of an almost contact metric manifold $(N,\phi,\xi,\eta,g)$ with the semi-slant function $\theta$.

\begin{enumerate}

\item Given a point $p\in M$, if $\xi(p)\in T_p M$, then $\xi(p)$ should belong to $(\mathcal{D}_1)_p$ (i.e., $\xi(p)\in (\mathcal{D}_1)_p$).

If not, we can induce contradiction as follows:

Assume that $\xi(p) = X + Y$ for some $X\in (\mathcal{D}_1)_p$ and some nonzero $Y\in (\mathcal{D}_2)_p$.
Then $0 = \phi \xi(p) = \phi X + \phi Y$ with $\phi X\in (\mathcal{D}_1)_p$ and $\phi Y\in (\mathcal{D}_2)_p \oplus T_p M^{\perp}$ so that
$\phi X = 0$ and $\phi Y = 0$. Since $g(X, Y) = 0$ and $\ker \phi = < \xi >$, we must have $X = 0$ and $Y = \xi(p)$.
But $\theta(Y) = \theta(\xi(p))$ is not defined, contradiction.

\item Let $(\overline{\mathcal{D}}_1)_p := \{ X\in (\mathcal{D}_1)_p \mid g(X, \xi(p)) = 0 \}$ for $p\in M$.

Then we have either $(\mathcal{D}_1)_p = (\overline{\mathcal{D}}_1)_p$ or $(\mathcal{D}_1)_p = < \xi(p) > \oplus (\overline{\mathcal{D}}_1)_p$.

We can check this as follows:

Since $\phi(\mathcal{D}_1) \subset \mathcal{D}_1$, we get
$$
\phi((\overline{\mathcal{D}}_1)_p) \subset (\mathcal{D}_1)_p \ \text{and} \ g(\phi((\overline{\mathcal{D}}_1)_p), \xi(p)) = 0
$$
so that $\phi((\overline{\mathcal{D}}_1)_p) \subset (\overline{\mathcal{D}}_1)_p$ implies $\phi((\overline{\mathcal{D}}_1)_p) = (\overline{\mathcal{D}}_1)_p$.
Thus, we can choose an orthonormal basis $\{ Z_1,\phi Z_1, \cdots, Z_k,\phi Z_k \}$ of $(\overline{\mathcal{D}}_1)_p$.
Assume that $(\mathcal{D}_1)_p \neq (\overline{\mathcal{D}}_1)_p$. Then there is a vector $Z = a\xi(p) + X\in (\mathcal{D}_1)_p$ with $a \neq 0$
and $g(\xi(p), X) = 0$. We know $\phi Z = \phi X\in (\mathcal{D}_1)_p$ and $g(\phi X, \xi(p)) = 0$ so that $\phi X\in (\overline{\mathcal{D}}_1)_p$ implies
$\displaystyle{\phi X = \sum_{i=1}^k (a_iZ_i + a_{k+i}\phi Z_i)}$ for some $a_i\in \mathbb{R}$, $1\leq i\leq 2k$.

Hence, $\displaystyle{-X = \phi^2 X = \sum_{i=1}^k (-a_{k+i}Z_i + a_i\phi Z_i) \in (\overline{\mathcal{D}}_1)_p \subset (\mathcal{D}_1)_p}$,
which implies $\frac{1}{a}(Z - X) = \xi(p) \in (\mathcal{D}_1)_p$. Therefore, the result follows.

\item From (2), we have either $\mathcal{D}_1 = \overline{\mathcal{D}}_1$ or $\mathcal{D}_1 = < \xi > \oplus \overline{\mathcal{D}}_1$,
where $\displaystyle{ \overline{\mathcal{D}}_1 := \bigcup_{p\in M} (\overline{\mathcal{D}}_1)_p }$.

If not, then we can choose a $C^{\infty}$-curve $c : (-\epsilon, \epsilon) \mapsto M$ for sufficiently small $\epsilon > 0$ such that
either $(\mathcal{D}_1)_{c(0)} = (\overline{\mathcal{D}}_1)_{c(0)}$ and $(\mathcal{D}_1)_{c(t)} = < \xi(c(t)) > \oplus (\overline{\mathcal{D}}_1)_{c(t)}$
for $t\in (-\epsilon, \epsilon) - \{ 0 \}$ or $(\mathcal{D}_1)_{c(0)} = < \xi(c(0)) > \oplus (\overline{\mathcal{D}}_1)_{c(0)}$
and  $(\mathcal{D}_1)_{c(t)} = (\overline{\mathcal{D}}_1)_{c(t)}$ for $t\in (-\epsilon, \epsilon) - \{ 0 \}$.

Take an orthonormal frame $\{ X_1(t),X_2(t),\cdots,X_l(t) \}$ of $\mathcal{D}_1$ along $c$. At the first case, we obtain
\begin{equation}\label{eq: sslant01}
\xi(c(t)) = \sum_{i=1}^l a_i(t) X_i(t)
\end{equation}
for some $a_i(t)\in \mathbb{R}$, $1\leq i \leq l$, $t\in (-\epsilon, \epsilon) - \{ 0 \}$.
Since $\xi$ is a $C^{\infty}$-vector field on $N$, we can obtain  the $C^{\infty}$-extension of
right hand side of (\ref {eq: sslant01}) along $c$. But $\xi(c(0))\notin (\mathcal{D}_1)_{c(0)}$ and
$\displaystyle{\sum_{i=1}^l a_i(0) X_i(0)}\in (\mathcal{D}_1)_{c(0)}$ with
$\displaystyle{ a_i(0) := \lim_{t\rightarrow 0} a_i(t)}$, $1\leq i \leq l$, contradiction.
In a similar way, we can also induce contradiction at the second case.

\item From (1), we get $\xi(p)\notin (\mathcal{D}_2)_p$ for any $p\in M$.

\item If $\theta : M \mapsto (0, \frac{\pi}{2})$, then $M$ is said to be {\em proper}.

\end{enumerate}
\end{remark}

Let $M$ be a pointwise semi-slant submanifold of an almost contact metric manifold $(N,\phi,\xi,\eta,g)$.
Then there is a distribution $\mathcal{D}_1 \subset TM$ on $M$ such that
$$
TM = \mathcal{D}_1 \oplus \mathcal{D}_2, \quad \phi(\mathcal{D}_1) \subset \mathcal{D}_1,
$$
and at each given point $p\in M$ the angle $\theta = \theta(X)$ between $\phi X$ and the space $(\mathcal{D}_2)_p$ is constant for nonzero
$X\in (\mathcal{D}_2)_p$, where $\mathcal{D}_2$ is the orthogonal complement of $\mathcal{D}_1$ in $TM$.

For $X\in \Gamma(TM)$, we write
\begin{equation}\label{eq: sslant02}
X = PX + QX,
\end{equation}
where $PX\in \Gamma(\mathcal{D}_1)$ and $QX\in \Gamma(\mathcal{D}_2)$.

For $X\in \Gamma(TM)$, we have
\begin{equation}\label{eq: sslant03}
\phi X = TX + FX,
\end{equation}
where $TX\in \Gamma(TM)$ and $FX\in \Gamma(TM^{\perp})$.

For $Z\in \Gamma(TM^{\perp})$, we get
\begin{equation}\label{eq: sslant04}
\phi Z = tZ + fZ,
\end{equation}
where $tZ\in \Gamma(TM)$ and $fZ\in \Gamma(TM^{\perp})$.

Denote by $(TN)|_M$ the restriction of $TN$ to $M$ (i.e., $(TN)|_M = TM \oplus TM^{\perp}$).

For $U\in \Gamma((TN)|_M)$, we write
\begin{equation}\label{eq: sslant05}
U = \mathcal{H}U + \mathcal{V}U,
\end{equation}
where $\mathcal{H}U\in \Gamma(TM)$ and $\mathcal{V}U\in \Gamma(TM^{\perp})$.

Hence,
\begin{eqnarray}
& & T(\mathcal{D}_1) \subset \mathcal{D}_1, F(\mathcal{D}_1) = 0, T(\mathcal{D}_2) \subset \mathcal{D}_2,
t(TM^{\perp}) \subset \mathcal{D}_2,   \label{eq: sslant06}     \\
& & T^2 + tF = -I + \eta\otimes \mathcal{H}(\xi) \ \text{and} \ FT + fF = \eta\otimes \mathcal{V}(\xi) \ \text{on} \ TM   \label{eq: sslant07}      \\
& & Tt + tf = \eta\otimes \mathcal{H}(\xi) \ \text{and} \ Ft + f^2 = -I +  \eta\otimes \mathcal{V}(\xi) \ \text{on} \ TM^{\perp}     \label{eq: sslant08}
\end{eqnarray}
Then we obtain
\begin{equation}\label{eq: sslant09}
TM^{\perp} = F\mathcal{D}_2 \oplus \mu,
\end{equation}
where $\mu$ is the orthogonal complement of $F\mathcal{D}_2$ in $TM^{\perp}$.

For $X,Y\in \Gamma(TM)$, we define
\begin{eqnarray}
&  & (\nabla_X T)Y := \nabla_X (TY) - T\nabla_X Y,   \label{eq: sslant010}    \\
&  &  (D_X F)Y := D_X (FY) - F\nabla_X Y.    \label{eq: sslant011}
\end{eqnarray}
The tensors $T$ and $F$ are called {\em parallel} if $\nabla T = 0$ and $\nabla F = 0$, respectively.

In the same way to Lemma \ref {lem: top01}, we have

\begin{lemma}\label{lem: sslant01}
Let $M$ be a pointwise semi-slant submanifold of a cosymplectic manifold $(N,\phi,\xi,\eta,g)$.
Then we obtain
\begin{enumerate}
\item
\begin{eqnarray}
&  &(\nabla_X T) Y = A_{FY} X + th(X,Y),  \label{eq: sslant012}   \\
&  &(D_X F) Y = -h(X,TY) + fh(X,Y)   \label{eq: sslant013}
\end{eqnarray}
for $X,Y\in \Gamma(TM)$.

\item
\begin{eqnarray}
&  &-TA_Z X + tD_X Z = \nabla_X (tZ) - A_{fZ} X,  \label{eq: sslant014}  \\
&  &-FA_Z X + fD_X Z = h(X,tZ) + D_X (fZ)   \label{eq: sslant015}
\end{eqnarray}
for $X\in \Gamma(TM)$ and $Z\in \Gamma(TM^{\perp})$.
\end{enumerate}
\end{lemma}

\begin{proposition}\label{prop: sslant01}
Let $M$ be a pointwise semi-slant submanifold of an almost contact metric manifold $(N,\phi,\xi,\eta,g)$.
Assume that either $\mathcal{D}_2 \subset \ker \eta$ or $\mu \subset \ker \eta$.

Then $\mu$ is $\phi$-invariant (i.e., $\phi \mu \subset \mu$).
\end{proposition}

\begin{proof}
Given $Y\in \Gamma(\mu)$ and $X\in \Gamma(TM)$ with $X = X_1 + X_2$, $X_1\in \Gamma(\mathcal{D}_1)$, $X_2\in \Gamma(\mathcal{D}_2)$,
we have
$$
g(X, \phi Y) = -g(\phi X, Y) = -g(\phi X_1 + \phi X_2, Y) = 0
$$
so that
\begin{equation}\label{eq: sslant016}
\phi \mu \subset TM^{\perp}.
\end{equation}
Given $Y\in \Gamma(\mu)$ and $X\in \Gamma(F\mathcal{D}_2)$ with $X = FX'$ for some $X'\in \Gamma(\mathcal{D}_2)$, by using (\ref {eq: sslant07})
and the hypothesis, we get
\begin{align*}
g(X, \phi Y)
&= -g(\phi X, Y) = -g(fFX', Y)   \\
&= g(FTX' - \eta(X') \mathcal{V}(\xi), Y)  \\
&= -\eta(X')\cdot \eta(Y) = 0.
\end{align*}
With (\ref {eq: sslant016}), it implies $\phi \mu \subset \mu$.
\end{proof}

In a similar way to Proposition 3.9 of \cite {P4}, we have

\begin{lemma}\label{lem: sslant02}
Let $M$ be a pointwise semi-slant submanifold of an almost contact metric manifold $(N,\phi,\xi,\eta,g)$ with the semi-slant function $\theta$.

Then
\begin{equation}\label{eq: sslant017}
g((T^2 + \cos^2 \theta (I - \eta\otimes \xi))(X), Y) = 0 \quad \text{for} \ X,Y\in \Gamma(\mathcal{D}_2).
\end{equation}
\end{lemma}

\begin{proof}
We will prove this at each point of $M$.

Gven a point $p\in M$, if $X\in (\mathcal{D}_2)_p$ is vanishing, then done!
Given a nonzero $X\in (\mathcal{D}_2)_p$, we obtain
\begin{equation}\label{eq: sslant018}
\cos \theta(p) = \frac{g(\phi X, TX)}{||\phi X|| \ ||TX||} = \frac{||TX||}{||\phi X||}
\end{equation}
so that $\cos^2 \theta(p) g(\phi X, \phi X) = g(TX, TX) = -g(T^2X, X)$.
Substituting $X$ by $X+Y$, $Y\in (\mathcal{D}_2)_p$, at the above equation, we induce
\begin{equation}\label{eq: sslant019}
g((T^2 + \cos^2 \theta (I - \eta\otimes \xi))(X), Y) + g(X, (T^2 + \cos^2 \theta (I - \eta\otimes \xi))(Y)) = 0.
\end{equation}
But $T^2 + \cos^2 \theta (I - \eta\otimes \xi)$ is also symmetric so that
\begin{equation*}\label{eq: sslant020}
g((T^2 + \cos^2 \theta (I - \eta\otimes \xi))(X), Y) = 0.
\end{equation*}
\end{proof}

\begin{remark}\label{rem: sslant02}
Let $M$ be a pointwise semi-slant submanifold of an almost contact metric manifold $(N,\phi,\xi,\eta,g)$ with the semi-slant function $\theta$.
Assume that either $\xi$ is tangent to $M$ or $\xi$ is normal to $M$.

\begin{enumerate}

\item By using (\ref {eq: sslant017}) and Remark \ref {rem: sslant01} (1), we get
\begin{equation}\label{eq: sslant021}
T^2 X = -\cos^2 \theta \cdot X \quad \text{for} \ X\in \Gamma(\mathcal{D}_2).
\end{equation}

\item By (\ref {eq: sslant021}), we obtain
\begin{equation}\label{eq: sslant022}
g(TX, TY) = \cos^2 \theta g(X, Y),
\end{equation}
\begin{equation}\label{eq: sslant023}
g(FX, FY) = \sin^2 \theta g(X, Y),
\end{equation}
for $X,Y\in \Gamma(\mathcal{D}_2)$.

\item At each given point $p\in M$ with $0\leq \theta(p) < \frac{\pi}{2}$, by using (\ref {eq: sslant022}), we can choose an orthonormal basis
$\{ X_1,\sec \theta TX_1,\cdots,X_k,\sec \theta TX_k \}$ of $(\mathcal{D}_2)_p$.
\end{enumerate}
\end{remark}

\section{Distributions}\label{dist}

In this section we consider distributions $\mathcal{D}_1$ and $\mathcal{D}_2$ and deal with the notion of totally umbilic submanifolds.

Notice that if $N = (N,\phi,\xi,\eta,g)$ is Sasakian, then from Theorem \ref {thm: ang02},  there does not exist a proper
pointwise semi-slant submanifold $M$ of $N$ such that $\xi$ is normal to $M$.

\begin{lemma}\label{lem: dist01}
Let $M$ be a proper pointwise semi-slant submanifold of an almost contact metric manifold $(N,\phi,\xi,\eta,g)$. Assume that
$\xi$ is tangent to $M$  and $N$ is one of the following three manifolds: cosymplectic, Sasakian, Kenmotsu.

Then the distribution $\mathcal{D}_1$ is integrable if and only if
\begin{equation}\label{eq: dist01}
g(h(X, \phi Y) - h(Y, \phi X), FZ) = 0
\end{equation}
for $X,Y\in \Gamma(\mathcal{D}_1)$ and $Z\in \Gamma(\mathcal{D}_2)$.
\end{lemma}

\begin{proof}
We will only give its proof when $N$ is Sasakian. For the other cases, we can show them in the same way.

Given $X,Y\in \Gamma(\mathcal{D}_1)$ and $Z\in \Gamma(\mathcal{D}_2)$, by using Remark \ref {rem: sslant01} and (\ref {eq: struc10}), we obtain
\begin{align*}
&g([X,Y], Z)   \\
&= g(\phi [X,Y], \phi Z) + \eta([X,Y]) \eta(Z)   \\
&= g(\phi(\overline{\nabla}_X Y - \overline{\nabla}_Y X), TZ + FZ)  \\
&= -g(\overline{\nabla}_X Y - \overline{\nabla}_Y X, T^2Z + FTZ)  \\
&\ + g(h(X, \phi Y) - h(Y, \phi X) - (g(X, Y)\xi - \eta(Y)X - g(Y, X)\xi + \eta(X)Y), FZ)  \\
&= \cos^2 \theta g([X,Y], Z) + g(h(X, \phi Y) - h(Y, \phi X), FZ)
\end{align*}
so that
$$
\sin^2 \theta g([X,Y], Z) = g(h(X, \phi Y) - h(Y, \phi X), FZ).
$$
Therefore, we get the result.
\end{proof}

In the same way to Lemma \ref {lem: dist01}, we obtain

\begin{lemma}\label{lem: dist001}
Let $M$ be a proper pointwise semi-slant submanifold of an almost contact metric manifold $(N,\phi,\xi,\eta,g)$.
Assume that  $\xi$ is normal to $M$ and $N$ is one of the following two manifolds: cosymplectic, Kenmotsu.

Then the distribution $\mathcal{D}_1$ is integrable if and only if
\begin{equation}\label{eq: dist001}
g(h(X, \phi Y) - h(Y, \phi X), FZ) = 0
\end{equation}
for $X,Y\in \Gamma(\mathcal{D}_1)$ and $Z\in \Gamma(\mathcal{D}_2)$.
\end{lemma}

\begin{lemma}\label{lem: dist02}
Let $M$ be a proper pointwise semi-slant submanifold of an almost contact metric manifold $(N,\phi,\xi,\eta,g)$.
Assume that $\xi$ is normal to $M$ and $N$ is one of the following two manifolds: cosymplectic, Kenmotsu.

Then the distribution $\mathcal{D}_2$ is integrable if and only if
\begin{equation}\label{eq: dist02}
g(A_{FTW} Z - A_{FTZ} W, X) = g(A_{FW} Z - A_{FZ} W, \phi X)
\end{equation}
for $X\in \Gamma(\mathcal{D}_1)$ and $Z,W\in \Gamma(\mathcal{D}_2)$.
\end{lemma}

\begin{proof}
We only give its proof when $N$ is Kenmotsu.

Given $X\in \Gamma(\mathcal{D}_1)$ and $Z,W\in \Gamma(\mathcal{D}_2)$, by using (\ref {eq: struc12}) and  Remark \ref {rem: sslant02}, we get
\begin{align*}
&g([Z,W], X)   \\
&= g(\phi [Z,W], \phi X) + \eta([Z,W]) \eta(X)   \\
&= g(\phi(\overline{\nabla}_Z W - \overline{\nabla}_W Z), \phi X)  \\
&= g(\overline{\nabla}_Z (TW + FW) - \overline{\nabla}_W (TZ + FZ), \phi X)  \\
&\ - g(g(\phi Z, W)\xi - \eta(W) \phi Z - g(\phi W, Z)\xi + \eta(Z) \phi W, \phi X)   \\
&= -g(\overline{\nabla}_Z (T^2 W + FTW) - \overline{\nabla}_W (T^2 Z + FTZ), X)  \\
&\ + g(g(\phi Z, TW)\xi - \eta(TW) \phi Z - g(\phi W, TZ)\xi + \eta(TZ) \phi W, X)      \\
&\ + g(A_{FZ} W - A_{FW} Z, \phi X)      \\
&= \cos^2 \theta g([Z,W], X) + g(A_{FTW} Z - A_{FTZ} W, X) + g(A_{FZ} W - A_{FW} Z, \phi X)
\end{align*}
so that
$$
\sin^2 \theta g([Z,W], X) = g(A_{FTW} Z - A_{FTZ} W, X) + g(A_{FZ} W - A_{FW} Z, \phi X).
$$
Therefore, the result follows.
\end{proof}

\begin{lemma}\label{lem: dist03}
Let $M$ be a proper pointwise semi-slant submanifold of an almost contact metric manifold $(N,\phi,\xi,\eta,g)$.
Assume that $\xi$ is tangent to $M$ and $N$ is one of the following two manifolds: cosymplectic, Kenmotsu.

Then the distribution $\mathcal{D}_2$ is integrable if and only if
\begin{equation}\label{eq: dist03}
g(A_{FTW} Z - A_{FTZ} W, X) = g(A_{FW} Z - A_{FZ} W, \phi X)
\end{equation}
for $X\in \Gamma(\mathcal{D}_1)$ and $Z,W\in \Gamma(\mathcal{D}_2)$.
\end{lemma}

\begin{proof}
We will show it when $N$ is Kenmotsu.

Given $X\in \Gamma(\mathcal{D}_1)$ and $Z,W\in \Gamma(\mathcal{D}_2)$, from the proof of Lemma \ref {lem: dist02}, we have
\begin{eqnarray}
\sin^2 \theta g([Z,W], X) &=& \eta([Z,W]) \eta(X) + g(A_{FTW} Z - A_{FTZ} W, X)   \label{eq: dist04}  \\
 & &\ + g(A_{FZ} W - A_{FW} Z, \phi X).   \nonumber
\end{eqnarray}
Replacing $X$ by $\xi$ at (\ref {eq: dist04}), by using (\ref {eq: shape}) and (\ref {eq: struc13}), we get
\begin{align*}
-\cos^2 \theta \eta([Z,W])
&= g(h(Z, \xi), FTW) - g(h(W, \xi), FTZ)   \\
&= g(\overline{\nabla}_Z \xi, FTW) - g(\overline{\nabla}_W \xi, FTZ)  \\
&= g(Z- \eta(Z)\xi, FTW) - g(W- \eta(W)\xi, FTZ)  \\
&= 0
\end{align*}
so that $\eta([Z,W]) = 0$.

Hence, the result follows.
\end{proof}

\begin{remark}
For the case when both $N$ is Sasakian and $\xi$ is tangent to $M$, confer Proposition 5.4 of \cite {CCFF0}.
\end{remark}

\begin{theorem}\label{thm: dist01}
Let $M$ be a proper pointwise semi-slant submanifold of an almost contact metric manifold $(N,\phi,\xi,\eta,g)$.
Assume that $\xi$ is tangent to $M$ and $N$ is one of the following three manifolds: cosymplectic, Sasakian, Kenmotsu.

Then the distribution $\mathcal{D}_1$ defines a totally geodesic foliation if and only if
\begin{equation}\label{eq: dist05}
g(A_{FZ} \phi X - A_{FTZ} X, Y) = 0
\end{equation}
for $X,Y\in \Gamma(\mathcal{D}_1)$ and $Z\in \Gamma(\mathcal{D}_2)$.
\end{theorem}

\begin{proof}
We will give its proof when $N$ is Kenmotsu.

Given $X,Y\in \Gamma(\mathcal{D}_1)$ and $Z\in \Gamma(\mathcal{D}_2)$, by using Remark \ref {rem: sslant01}, (\ref {eq: struc12}),
and  Remark \ref {rem: sslant02}, we obtain
\begin{align*}
&g(\nabla_Y X, Z)    \\
&= g(\phi \overline{\nabla}_Y X, \phi Z) + \eta(\overline{\nabla}_Y X) \eta(Z)  \\
&= g(\phi \overline{\nabla}_Y X, TZ + FZ)  \\
&= -g(\overline{\nabla}_Y X, T^2Z + FTZ)  \\
&\ + g(\overline{\nabla}_Y \phi X - (g(\phi Y, X)\xi - \eta(X) \phi Y), FZ)    \\
&= \cos^2 \theta g(\nabla_Y X, Z) - g(A_{FTZ} X, Y) + g(A_{FZ} \phi X, Y)
\end{align*}
so that
$$
\sin^2 \theta g(\nabla_Y X, Z) = g(A_{FZ} \phi X - A_{FTZ} X, Y).
$$
Therefore, we obtain the result.
\end{proof}

In the same way to Theorem \ref {thm: dist01}, we get

\begin{theorem}\label{thm: dist001}
Let $M$ be a proper pointwise semi-slant submanifold of an almost contact metric manifold $(N,\phi,\xi,\eta,g)$.
Assume that $\xi$ is normal to $M$ and $N$ is one of the following two manifolds: cosymplectic, Kenmotsu.

Then the distribution $\mathcal{D}_1$ defines a totally geodesic foliation if and only if
\begin{equation}\label{eq: dist005}
g(A_{FZ} \phi X - A_{FTZ} X, Y) = 0
\end{equation}
for $X,Y\in \Gamma(\mathcal{D}_1)$ and $Z\in \Gamma(\mathcal{D}_2)$.
\end{theorem}

\begin{theorem}\label{thm: dist02}
Let $M$ be a proper pointwise semi-slant submanifold of an almost contact metric manifold $(N,\phi,\xi,\eta,g)$.
Assume that $\xi$ is normal to $M$ and $N$ is one of the following two manifolds: cosymplectic, Kenmotsu.

Then the distribution $\mathcal{D}_2$ defines a totally geodesic foliation if and only if
\begin{equation}\label{eq: dist06}
g(A_{FZ} \phi X - A_{FTZ} X, W) = 0
\end{equation}
for $X\in \Gamma(\mathcal{D}_1)$ and $Z,W\in \Gamma(\mathcal{D}_2)$.
\end{theorem}

\begin{proof}
We give its proof when $N$ is Kenmotsu.

Given $X\in \Gamma(\mathcal{D}_1)$ and $Z,W\in \Gamma(\mathcal{D}_2)$, by using (\ref {eq: struc12}) and Remark \ref {rem: sslant02},  we get
\begin{align*}
&g(\nabla_W Z, X)    \\
&= g(\phi \overline{\nabla}_W Z, \phi X) + \eta(\overline{\nabla}_W Z) \eta(X)  \\
&= g(\overline{\nabla}_W (TZ + FZ) - (g(\phi W, Z)\xi - \eta(Z)\phi W), \phi X)  \\
&= -g(\overline{\nabla}_W (T^2 Z + FTZ) - (g(\phi W, TZ)\xi - \eta(TZ)\phi W), X) - g(A_{FZ} W, \phi X)  \\
&= \cos^2 \theta g(\nabla_W Z, X) + g(A_{FTZ} W, X) - g(A_{FZ} W, \phi X)
\end{align*}
so that
$$
\sin^2 \theta g(\nabla_W Z, X) = g(A_{FTZ} X - A_{FZ} \phi X, W).
$$
Therefore, the result follows.
\end{proof}

In a similar way, we have

\begin{theorem}\label{thm: dist03}
Let $M$ be a proper pointwise semi-slant submanifold of an almost contact metric manifold $(N,\phi,\xi,\eta,g)$.
Assume that $\xi$ is tangent to $M$
\begin{enumerate}
\item If $N$ is one of the following two manifolds: cosymplectic, Sasakian, then $\mathcal{D}_2$ defines a totally geodesic foliation if and only if
\begin{equation}\label{eq: dist07}
g(A_{FZ} \phi X - A_{FTZ} X, W) = 0
\end{equation}
for $X\in \Gamma(\mathcal{D}_1)$ and $Z,W\in \Gamma(\mathcal{D}_2)$.

\item If $N$ is Kenmotsu, then $\mathcal{D}_2$ defines a totally geodesic foliation if and only if
\begin{equation}\label{eq: dist08}
g(A_{FZ} \phi X - A_{FTZ} X, W) + \sin^2 \theta \eta(X) g(W, Z) = 0
\end{equation}
for $X\in \Gamma(\mathcal{D}_1)$ and $Z,W\in \Gamma(\mathcal{D}_2)$.
\end{enumerate}
\end{theorem}

\begin{proof}
We only give its proof when $N$ is Sasakian. For the other cases, we can show them in the same way.

Given $X\in \Gamma(\mathcal{D}_1)$ and $Z,W\in \Gamma(\mathcal{D}_2)$, by using (\ref {eq: struc10}) and Remark \ref {rem: sslant02},  we obtain
\begin{align*}
&g(\nabla_W Z, X)    \\
&= g(\phi \overline{\nabla}_W Z, \phi X) + \eta(\overline{\nabla}_W Z) \eta(X)  \\
&= g(\overline{\nabla}_W (TZ + FZ) - (g(W, Z)\xi - \eta(Z)W), \phi X) + \eta(\nabla_W Z) \eta(X)  \\
&= -g(\overline{\nabla}_W (T^2 Z + FTZ) - (g(W, TZ)\xi - \eta(TZ)W), X)    \\
&\ - g(A_{FZ} W, \phi X) + \eta(\nabla_W Z) \eta(X) \\
&= \cos^2 \theta g(\nabla_W Z, X) + g(A_{FTZ} W, X) + g(W, TZ)\eta(X)  \\
&\ - g(A_{FZ} W, \phi X) + \eta(\nabla_W Z) \eta(X)
\end{align*}
so that
\begin{eqnarray}
\sin^2 \theta g(\nabla_W Z, X) &=& g(A_{FTZ} X - A_{FZ} \phi X, W)  \label{eq: dist08}  \\
& &+ g(W, TZ)\eta(X) + \eta(\nabla_W Z) \eta(X).  \nonumber
\end{eqnarray}
Replacing $X$ by $\xi$ at (\ref {eq: dist08}), we get
$$
\sin^2 \theta \eta(\nabla_W Z) = g(h(W, \xi), FTZ) + g(W, TZ) + \eta(\nabla_W Z)
$$
so that by using (\ref {eq: struc11}) and Remark \ref {rem: sslant02},
\begin{align*}
-\cos^2 \theta \eta(\nabla_W Z)
&= g(\overline{\nabla}_W \xi, FTZ) + g(W, TZ)   \\
&= g(-\phi W, FTZ) + g(W, TZ)  \\
&= -\sin^2 \theta g(W, TZ) + g(W, TZ)  \\
&= \cos^2 \theta g(W, TZ),
\end{align*}
which implies $\eta(\nabla_W Z) = -g(W, TZ)$.

Hence, from (\ref {eq: dist08}),
$$
\sin^2 \theta g(\nabla_W Z, X) = g(A_{FTZ} X - A_{FZ} \phi X, W).
$$
Therefore, the result follows.
\end{proof}

Using Theorem \ref {thm: dist01} and Theorem \ref {thm: dist03}, we obtain

\begin{corollary}
Let $M$ be a proper pointwise semi-slant submanifold of an almost contact metric manifold $(N,\phi,\xi,\eta,g)$.
Assume that  $\xi$ is tangent to $M$ and $N$ is one of the following two manifolds: cosymplectic, Sasakian.

Then $M$ is locally a Riemannian product manifold of $M_1$ and $M_2$ if and only if
\begin{equation}\label{eq: dist09}
A_{FZ} \phi X = A_{FTZ} X
\end{equation}
for $X\in \Gamma(\mathcal{D}_1)$ and $Z\in \Gamma(\mathcal{D}_2)$,
where $M_1$ and $M_2$ are integral manifolds of $\mathcal{D}_1$ and $\mathcal{D}_2$, respectively.
\end{corollary}

Using Theorem \ref {thm: dist001} and Theorem \ref {thm: dist02}, we also obtain

\begin{corollary}
Let $M$ be a proper pointwise semi-slant submanifold of an almost contact metric manifold $(N,\phi,\xi,\eta,g)$.
Assume that  $\xi$ is normal to $M$ and $N$ is one of the following two manifolds: cosymplectic, Kenmotsu.

Then $M$ is locally a Riemannian product manifold of $M_1$ and $M_2$ if and only if
\begin{equation}\label{eq: dist009}
A_{FZ} \phi X = A_{FTZ} X
\end{equation}
for $X\in \Gamma(\mathcal{D}_1)$ and $Z\in \Gamma(\mathcal{D}_2)$,
where $M_1$ and $M_2$ are integral manifolds of $\mathcal{D}_1$ and $\mathcal{D}_2$, respectively.
\end{corollary}

Let $M$ be a submanifold of a Riemannian manifold $(N, g)$. We call $M$ a {\em totally umbilic submanifold} of $(N, g)$ if
\begin{equation}\label{eq: dist010}
h(X, Y) = g(X, Y)H \quad \text{for} \ X,Y\in \Gamma(TM),
\end{equation}
where $H$ is the mean curvature vector field of $M$ in $N$.

\begin{lemma}\label{lem: dist04}
Let $M$ be a pointwise semi-slant totally umbilic submanifold of an almost contact metric manifold $(N,\phi,\xi,\eta,g)$.
Assume that $\xi$ is tangent to $M$ and $N$ is one of the following three manifolds: cosymplectic, Sasakian, Kenmotsu.

Then
\begin{equation}\label{eq: dist011}
H\in \Gamma(F\mathcal{D}_2).
\end{equation}
\end{lemma}

\begin{proof}
We give its proof when $N$ is Kenmotsu.

Since $\xi$ is tangent to $M$, by Proposition \ref {prop: sslant01}, $\mu$ is $\phi$-invariant (i.e., $\phi(\mu) = \mu$).
Given $X,Y\in \Gamma(\mathcal{D}_1)$ and $Z\in \Gamma(\mu)$, we have
\begin{align*}
&\nabla_X \phi Y + h(X, \phi Y)    \\
&= \overline{\nabla}_X \phi Y     \\
&= g(\phi X, Y)\xi - \eta(Y)\phi X + \phi \overline{\nabla}_X Y  \\
&= g(\phi X, Y)\xi - \eta(Y)\phi X + T\nabla_X Y + F\nabla_X Y + th(X, Y) + fh(X, Y)
\end{align*}
so that by taking the inner product of both sides with $Z$,
\begin{equation}\label{eq: dist012}
g(h(X, \phi Y), Z) = g(fh(X, Y), Z).
\end{equation}
From (\ref {eq: dist012}), by (\ref {eq: dist010}) we obtain
\begin{equation}\label{eq: dist013}
g(X, \phi Y) g(H, Z) = -g(X, Y) g(H, \phi Z).
\end{equation}
Interchanging the role of $X$ and $Y$,
\begin{equation}\label{eq: dist014}
g(Y, \phi X) g(H, Z) = -g(Y, X) g(H, \phi Z).
\end{equation}
Comparing (\ref {eq: dist013}) with (\ref {eq: dist014}), we have
$$
g(X, Y) g(H, \phi Z) = 0,
$$
which means $H\in \Gamma(F\mathcal{D}_2)$.
\end{proof}

Using Lemma \ref {lem: dist04}, we immediately obtain

\begin{corollary}
Let $M$ be a pointwise semi-slant totally umbilic submanifold of an almost contact metric manifold $(N,\phi,\xi,\eta,g)$ with the semi-slant function $\theta$.
Assume that $\xi$ is tangent to $M$ and $N$ is one of the following three manifolds: cosymplectic, Sasakian, Kenmotsu.

If $\theta = 0$ on $M$, then $M$ is a totally geodesic submanifold of $N$.
\end{corollary}

\section{Warped product submanifolds}\label{warped}

In this section we consider the non-existence of some type of warped product pointwise semi-slant submanifolds
and investigate the properties of some warped product pointwise semi-slant submanifolds.

\begin{theorem}
Let $N = (N,\phi,\xi,\eta,g)$ be an almost contact metric manifold
and $M = B\times_f \overline{F}$ a nontrivial warped product submanifold of $N$.
Assume that $\xi$ is normal to $M$ and $N$ is one of the following three manifolds: cosymplectic, Sasakian, Kenmotsu.

Then there does not exist a proper pointwise semi-slant submanifold $M$ of $N$ such that $\mathcal{D}_1 = T\overline{F}$ and $\mathcal{D}_2 = TB$.
\end{theorem}

\begin{proof}
If $N$ is Sasakian, then by Theorem \ref {thm: ang02}, it is obviously true.

We will prove it when $N$ is Kenmotsu. For the case of $N$ to be cosymplectic, we can prove it in the same way.

Suppose that there exists a proper pointwise semi-slant submanifold $M = B\times_f \overline{F}$ of $N$
such that $\mathcal{D}_1 = T\overline{F}$ and $\mathcal{D}_2 = TB$. We will induce contradiction.

Given $X,Y\in \Gamma(T\overline{F})$ and $Z\in \Gamma(TB)$, by using (\ref {eq: warp2}), (\ref {eq: struc12}), and Remark \ref {rem: sslant02},    we get
\begin{align*}
&Z(\ln f) g(X, Y)      \\
&= g(\overline{\nabla}_X Z, Y)   \\
&= g(\phi\overline{\nabla}_X Z, \phi Y) + \eta(\overline{\nabla}_X Z) \eta(Y) \\
&= g(\overline{\nabla}_X (TZ + FZ) - (g(\phi X, Z)\xi - \eta(Z) \phi X), \phi Y)  \\
&= g(\overline{\nabla}_X (TZ + FZ), \phi Y)  \\
&= -g(\overline{\nabla}_X (T^2Z + FTZ) - (g(\phi X, TZ)\xi - \eta(TZ)\phi X), Y) + g(\overline{\nabla}_X FZ, \phi Y)  \\
&= \cos^2 \theta g(\nabla_X Z, Y) + g(h(X, Y), FTZ) - g(h(X, \phi Y), FZ)
\end{align*}
so that
\begin{equation}\label{eq: warp01}
\sin^2 \theta Z(\ln f) g(X, Y) = g(h(X, Y), FTZ) - g(h(X, \phi Y), FZ).
\end{equation}
Interchanging the role of $X$ and $Y$, we have
\begin{equation}\label{eq: warp02}
\sin^2 \theta Z(\ln f) g(Y, X) = g(h(Y, X), FTZ) - g(h(Y, \phi X), FZ).
\end{equation}
Comparing (\ref {eq: warp01}) with (\ref {eq: warp02}), we obtain
\begin{equation}\label{eq: warp03}
g(h(X, \phi Y), FZ) =  g(h(Y, \phi X), FZ).
\end{equation}
On the other hand,
\begin{align*}
&g(h(X, \phi Y), FZ)      \\
&= g(A_{FZ} X, \phi Y)   \\
&= g(-\overline{\nabla}_X FZ, \phi Y) \\
&= g(-\overline{\nabla}_X (\phi Z - TZ), \phi Y)  \\
&= -g(g(\phi X, Z)\xi - \eta(Z)\phi X + \phi \overline{\nabla}_X Z, \phi Y) + g(\overline{\nabla}_X TZ, \phi Y)  \\
&= -g(\nabla_X Z, Y) + \eta(\overline{\nabla}_X Z) \eta(Y) + g(\nabla_X TZ, \phi Y)  \\
&= -Z(\ln f) g(X, Y) + TZ(\ln f) g(X, \phi Y).
\end{align*}
From (\ref {eq: warp03}), by using the above result, we obtain
\begin{equation}\label{eq: warp04}
TZ(\ln f) g(X, \phi Y) =  0.
\end{equation}
Replacing $Z$ by $\phi Z$ and $X$ by $\phi X$ at (\ref {eq: warp04}), by Remark \ref {rem: sslant02} we get
$$
\cos^2 \theta Z(\ln f) g(X, Y) = 0,
$$
which implies $Z(\ln f) = 0$
so that $f$ is constant, contradiction.
\end{proof}

\begin{theorem}
Let $N = (N,\phi,\xi,\eta,g)$ be an almost contact metric manifold
and $M = B\times_f \overline{F}$ a nontrivial warped product submanifold of $N$.
Assume that $\xi$ is tangent to $M$ and $N$ is one of the following three manifolds: cosymplectic, Sasakian, Kenmotsu.

Then there does not exist a proper pointwise semi-slant submanifold $M$ of $N$ such that $\mathcal{D}_1 = T\overline{F}$ and $\mathcal{D}_2 = TB$.
\end{theorem}

\begin{proof}
We will only give its proof when $N$ is Sasakian. For the other cases, we can show them in the same way.

Suppose that there exists a proper pointwise semi-slant submanifold $M = B\times_f \overline{F}$ of $N$ such that
$\mathcal{D}_1 = T\overline{F}$ and $\mathcal{D}_2 = TB$. We will also induce contradiction.

Given $X,Y\in \Gamma(T\overline{F})$ and $Z\in \Gamma(TB)$, by using (\ref {eq: warp2}), (\ref {eq: struc10}), Remark \ref {rem: sslant01},
and Remark \ref {rem: sslant02},    we have
\begin{align*}
&Z(\ln f) g(X, Y)      \\
&= g(\overline{\nabla}_X Z, Y)   \\
&= g(\phi\overline{\nabla}_X Z, \phi Y) + \eta(\overline{\nabla}_X Z) \eta(Y) \\
&= g(\overline{\nabla}_X (TZ + FZ) - (g(X, Z)\xi - \eta(Z) X), \phi Y) + Z(\ln f) \eta(X)\eta(Y) \\
&= g(\overline{\nabla}_X (TZ + FZ), \phi Y)  + Z(\ln f) \eta(X)\eta(Y) \\
&= -g(\overline{\nabla}_X (T^2Z + FTZ) - (g(X, TZ)\xi - \eta(TZ) X), Y)    \\
&\ + g(\overline{\nabla}_X FZ, \phi Y)  + Z(\ln f) \eta(X)\eta(Y) \\
&= \cos^2 \theta g(\nabla_X Z, Y) + g(h(X, Y), FTZ) - g(h(X, \phi Y), FZ)  + Z(\ln f) \eta(X)\eta(Y)
\end{align*}
so that
\begin{eqnarray}
\sin^2 \theta Z(\ln f) g(X, Y) &=& g(h(X, Y), FTZ) - g(h(X, \phi Y), FZ)  \label{eq: warp05}  \\
& &\ + Z(\ln f) \eta(X)\eta(Y).   \nonumber
\end{eqnarray}
Replacing $X$ and $Y$ by $\xi$ at (\ref {eq: warp05}), by using (\ref {eq: struc11}) we obtain
\begin{align*}
\cos^2 \theta Z(\ln f)
&= -g(h(\xi, \xi), FTZ)   \\
&= -g(\overline{\nabla}_\xi \xi, FTZ) \\
&= -g(-\phi \xi, FTZ)  \\
&= 0,
\end{align*}
which implies $Z(\ln f) = 0$ so that
 $f$ is constant, contradiction.
\end{proof}

Now, we will study nontrivial warped product pointwise semi-slant submanifold $M = B\times_f \overline{F}$ of
an almost contact metric manifold $N = (N,\phi,\xi,\eta,g)$ such that $\mathcal{D}_1 = TB$ and $\mathcal{D}_2 = T\overline{F}$.

\begin{lemma}\label{lem: warp01}
Let $M = B\times_f \overline{F}$ be a nontrivial warped product proper pointwise semi-slant submanifold of
an almost contact metric manifold $N = (N,\phi,\xi,\eta,g)$ such that $\mathcal{D}_1 = TB$ and $\mathcal{D}_2 = T\overline{F}$.
Assume that $N$ is one of the following three manifolds: cosymplectic, Sasakian, Kenmotsu.

Then we get
\begin{equation}\label{eq: warp06}
g(A_{FZ} W, X) = g(A_{FW} Z, X)
\end{equation}
for $X\in \Gamma(TB)$ and $Z,W\in \Gamma(T\overline{F})$.
\end{lemma}

\begin{proof}
We give its proof when $N$ is Kenmotsu.

Given $X\in \Gamma(TB)$ and $Z,W\in \Gamma(T\overline{F})$, by using (\ref {eq: struc12}), (\ref {eq: sslant06}), and (\ref {eq: warp2}), we obtain
\begin{align*}
&g(A_{FZ} W, X)   \\
&= g(A_{FZ} X, W)   \\
&= -g(\overline{\nabla}_X (\phi Z - TZ), W) \\
&= -g(g(\phi X, Z)\xi - \eta(Z)\phi X + \phi \overline{\nabla}_X Z, W) + g(\overline{\nabla}_X TZ, W)  \\
&= g(\overline{\nabla}_X Z, TW + FW) + g(\nabla_X TZ, W)  \\
&= g(X(\ln f) Z, TW) + g(A_{FW} Z, X) + g(X(\ln f) TZ, W)  \\
&= g(A_{FW} Z, X).
\end{align*}
\end{proof}

\begin{lemma}\label{lem: warp02}
Let $M = B\times_f \overline{F}$ be a nontrivial warped product proper pointwise semi-slant submanifold of
an almost contact metric manifold $N = (N,\phi,\xi,\eta,g)$ such that $\mathcal{D}_1 = TB$ and $\mathcal{D}_2 = T\overline{F}$.
\begin{enumerate}
\item If $N$ is cosymplectic, then
\begin{equation}\label{eq: warp07}
g(A_{FTZ} W, X) = -\phi X(\ln f) g(W, TZ) - \cos^2 \theta X(\ln f) g(W, Z)
\end{equation}
and
\begin{equation}\label{eq: warp08}
g(A_{FZ} W, \phi X) = (X - \eta(X)\xi)(\ln f) g(W, Z) -  \phi X(\ln f) g(TW, Z)
\end{equation}
for $X\in \Gamma(TB)$ and $Z,W\in \Gamma(T\overline{F})$.

\item If $N$ is Sasakian, then
\begin{eqnarray}
g(A_{FTZ} W, X) &=& -\eta(X) g(TZ, W) -\phi X(\ln f) g(W, TZ)  \label{eq: warp09}  \\
& & - \cos^2 \theta X(\ln f) g(W, Z)  \nonumber
\end{eqnarray}
and
\begin{equation}\label{eq: warp010}
g(A_{FZ} W, \phi X) = (X - \eta(X)\xi)(\ln f) g(W, Z) -  \phi X(\ln f) g(TW, Z)
\end{equation}
for $X\in \Gamma(TB)$ and $Z,W\in \Gamma(T\overline{F})$.

\item If $N$ is Kenmotsu, then
\begin{eqnarray}
g(A_{FTZ} W, X) &=& \cos^2 \theta \eta(X) (g(Z, W) - \eta(Z) \eta(W))  \label{eq: warp011} \\
& & -\phi X(\ln f) g(W, TZ) - \cos^2 \theta X(\ln f) g(W, Z)   \nonumber
\end{eqnarray}
and
\begin{equation}\label{eq: warp012}
g(A_{FZ} W, \phi X) = (X - \eta(X)\xi)(\ln f) g(W, Z) -  \phi X(\ln f) g(TW, Z)
\end{equation}
for $X\in \Gamma(TB)$ and $Z,W\in \Gamma(T\overline{F})$.
\end{enumerate}
\end{lemma}

\begin{proof}
We only give its proof when $N$ is Kenmotsu.

Given $X\in \Gamma(TB)$ and $Z,W\in \Gamma(T\overline{F})$, by using Lemma \ref {lem: warp01}, (\ref {eq: struc12}), Lemma \ref {lem: sslant02},
and (\ref {eq: warp2}), we have
\begin{align*}
&g(A_{FTZ} W, X)   \\
&= g(A_{FW} TZ, X)   \\
&= -g(\overline{\nabla}_{TZ} (\phi W - TW), X) \\
&= -g(g(\phi TZ, W)\xi - \eta(W)\phi TZ + \phi \overline{\nabla}_{TZ} W, X) + g(\overline{\nabla}_{TZ} TW, X)  \\
&= \cos^2 \theta \eta(X) g(Z -\eta(Z)\xi, W) + g(\overline{\nabla}_{TZ} W, \phi X) - g(TW, \overline{\nabla}_{TZ} X)  \\
&= \cos^2 \theta \eta(X) (g(Z, W) - \eta(Z)\eta(W)) -\phi X(\ln f) g(W, TZ) - \cos^2 \theta X(\ln f) g(W, Z).
\end{align*}
Replacing $TZ$ and $X$ by $Z$ and $\phi X$, respectively,
$$
g(A_{FZ} W, \phi X) = (X - \eta(X)\xi)(\ln f) g(W, Z) -  \phi X(\ln f) g(TW, Z).
$$
\end{proof}

To obtain some inequalities on nontrivial warped product proper pointwise semi-slant submanifolds of
cosymplectic, Sasakian, Kenmotsu manifolds in the next section, we need to have

\begin{lemma}\label{lem: warp03}
Let $M = B\times_f \overline{F}$ be a nontrivial warped product proper pointwise semi-slant submanifold of
an almost contact metric manifold $N = (N,\phi,\xi,\eta,g)$ such that $\mathcal{D}_1 = TB$ and $\mathcal{D}_2 = T\overline{F}$.
\begin{enumerate}
\item If $N$ is cosymplectic, then
\begin{equation}\label{eq: warp013}
g(h(X, Y), FZ) = 0
\end{equation}
and
\begin{equation}\label{eq: warp014}
g(h(X, W), FZ) = -\phi X(\ln f) g(W, Z) +  (X - \eta(X)\xi)(\ln f) g(W, TZ)
\end{equation}
for $X,Y\in \Gamma(TB)$ and $Z,W\in \Gamma(T\overline{F})$.

\item If $N$ is Sasakian, then
\begin{equation}\label{eq: warp015}
g(h(X, Y), FZ) = \eta(Z) g(X, Y)
\end{equation}
and
\begin{eqnarray}
g(h(X, W), FZ) &=& -\eta(X) g(FW, FZ) -\phi X(\ln f) g(W, Z)  \label{eq: warp016} \\
& & +  (X - \eta(X)\xi)(\ln f) g(W, TZ)   \nonumber
\end{eqnarray}
for $X,Y\in \Gamma(TB)$ and $Z,W\in \Gamma(T\overline{F})$.

\item If $N$ is Kenmotsu, then
\begin{equation}\label{eq: warp017}
g(h(X, Y), FZ) = \eta(Z) g(\phi X, Y)
\end{equation}
and
\begin{eqnarray}
g(h(X, W), FZ) &=& -\eta(X) \eta(W) \eta(FZ) -\phi X(\ln f) g(W, Z)  \label{eq: warp018} \\
& & +  (X - \eta(X)\xi)(\ln f) g(W, TZ)   \nonumber
\end{eqnarray}
for $X,Y\in \Gamma(TB)$ and $Z,W\in \Gamma(T\overline{F})$.
\end{enumerate}
\end{lemma}

\begin{proof}
We will give its proof when $N$ is Sasakian.

Given $X,Y\in \Gamma(TB)$ and $Z,W\in \Gamma(T\overline{F})$, by using (\ref {eq: struc10}) and (\ref {eq: warp2}), we get
\begin{align*}
&g(h(X, Y), FZ)   \\
&= g(\overline{\nabla}_X Y, \phi Z - TZ)   \\
&= -g(\phi\overline{\nabla}_X Y,  Z) - g(\overline{\nabla}_X Y, TZ) \\
&= -g(\overline{\nabla}_X \phi Y - (g(X, Y)\xi - \eta(Y)X),  Z) + g(Y, \overline{\nabla}_X TZ)  \\
&= g(\phi Y, X(\ln f)Z) + \eta(Z) g(X, Y) + g(Y, X(\ln f) TZ)  \\
&= \eta(Z) g(X, Y),
\end{align*}
which gives (\ref {eq: warp015}).

Replacing $X$ by $\phi X$ at (\ref {eq: warp010}), we obtain
$$
g(h(X, W), FZ) = \eta(X) \eta(A_{FZ} W) -\phi X(\ln f) g(W, Z) +  (X - \eta(X)\xi)(\ln f) g(W, TZ).
$$
But by using (\ref {eq: shape}) and (\ref {eq: struc11}),
\begin{align*}
\eta(A_{FZ} W)
&= g(A_{FZ} W, \xi)   \\
&= g(h(W, \xi), FZ)  \\
&= g(\overline{\nabla}_W \xi, FZ)  \\
&= g(-\phi W, FZ)  \\
&= -g(FW, FZ),
\end{align*}
which gives (\ref {eq: warp016}).
\end{proof}

\section{inequalities}\label{ineq}

We will consider inequalities for the squared norm of the second fundamental form in terms of a warping function and a semi-slant function
for a warped product submanifold in cosymplectic manifolds, Sasakian manifolds, and Kenmotsu manifolds.

Let $M = B\times_f \overline{F}$ be a $m$-dimensional nontrivial warped product proper pointwise semi-slant submanifold of
a $(2n+1)$-dimensional almost contact metric manifold $(N,\phi,\xi,\eta,g)$ with the semi-slant function $\theta$
such that $\mathcal{D}_1 = TB$, $\mathcal{D}_2 = T\overline{F}$, and $\xi$ is tangent to $M$.

Then by using Remark \ref {rem: sslant02} we can choose a local orthonormal frame
$\{ e_1,e_2,\cdots,$ $e_{2m_1+1},v_1,\cdots,v_{2m_2},w_1,\cdots,w_{2m_2},u_1,\cdots,u_{2r} \}$ of $TN$ such that
$\{ e_1,\cdots,e_{2m_1+1} \} \subset \Gamma(\mathcal{D}_1)$, $\{ v_1,\cdots,v_{2m_2} \} \subset \Gamma(\mathcal{D}_2)$,
$\{ w_1,\cdots,w_{2m_2} \} \subset \Gamma(F\mathcal{D}_2)$, $\{ u_1,\cdots,u_{2r} \} \subset \Gamma(\mu)$ with the following conditions:
\begin{enumerate}
\item $e_{m_1+i} = \phi e_i$, $1\leq i \leq m_1$, $e_{2m_1+1} = \xi$,

\item $v_{m_2+i} = \sec \theta Tv_i$, $1\leq i \leq m_2$,

\item $w_{i} = \csc \theta Fv_i$, $1\leq i \leq 2m_2$,

\item $u_{r+i} = \phi u_i$, $1\leq i \leq r$.
\end{enumerate}
So, we have $m = 2m_1 + 2m_2 + 1$ and $n = m_1 + 2m_2 + r$.

Using the above notations, we obtain

\begin{theorem}\label{thm: ineq01}
Let $M = B\times_f \overline{F}$ be a $m$-dimensional nontrivial warped product proper pointwise semi-slant submanifold of
a $(2n+1)$-dimensional Sasakian manifold $(N,\phi,\xi,\eta,g)$ with the semi-slant function $\theta$
such that $\mathcal{D}_1 = TB$, $\mathcal{D}_2 = T\overline{F}$, and $\xi$ is tangent to $M$.

Assume that $n = m_1 + 2m_2$.

Then we have
\begin{equation}\label{eq: ineq01}
||h||^2 \geq 4m_2 (\csc^2 \theta + \cot^2 \theta) ||\phi \nabla(\ln f)||^2 + 4m_2 \sin^2 \theta
\end{equation}
with equality holding if and only if $g(h(Z, W), V) = 0$ for $Z,W\in \Gamma(T\overline{F})$ and $V\in \Gamma(TM^{\perp})$.
\end{theorem}

\begin{proof}
Since $\mu = 0$, we get
\begin{align*}
||h||^2
&= \sum_{i,j=1}^{2m_1+1} g(h(e_i, e_j), h(e_i, e_j)) + \sum_{i,j=1}^{2m_2} g(h(v_i, v_j), h(v_i, v_j))   \\
&\ + 2 \sum_{i=1}^{2m_1+1} \sum_{j=1}^{2m_2} g(h(e_i, v_j), h(e_i, v_j))  \\
&= \sum_{i,j=1}^{2m_1+1}  \sum_{k=1}^{2m_2} g(h(e_i, e_j), w_k)^2 + \sum_{i,j=1}^{2m_2} \sum_{k=1}^{2m_2} g(h(v_i, v_j), w_k)^2  \\
&\ + 2 \sum_{i=1}^{2m_1+1} \sum_{j,k=1}^{2m_2} g(h(e_i, v_j), w_k)^2.
\end{align*}
By using Lemma \ref {lem: warp03} and Remark \ref {rem: sslant01}, we obtain
\begin{eqnarray}
||h||^2  &=& \sum_{i,j,k=1}^{2m_2} g(h(v_i, v_j), w_k)^2             \label{eq: ineq02} \\
& &+ 2\csc^2 \theta \sum_{i=1}^{2m_1+1} \sum_{j,k=1}^{2m_2} (-\eta(e_i) g(Fv_j, Fv_k)   \nonumber  \\
& & - \phi e_i(\ln f) g(v_j, v_k) + (e_i - \eta(e_i)\xi)(\ln f) g(v_j, Tv_k))^2     \nonumber  \\
&=& \sum_{i,j,k=1}^{2m_2} g(h(v_i, v_j), w_k)^2    \nonumber  \\
& &+ 2\csc^2 \theta \sum_{i=1}^{2m_1} \sum_{j,k=1}^{2m_2} (-\phi e_i(\ln f) \delta_{jk} + e_i(\ln f) g(v_j, Tv_k))^2 \nonumber  \\
& &+ 2\csc^2 \theta \sum_{j,k=1}^{2m_2} (-\sin^2 \theta \delta_{jk})^2  \nonumber  \\
&=& \sum_{i,j,k=1}^{2m_2} g(h(v_i, v_j), w_k)^2    \nonumber  \\
& &+ 2\csc^2 \theta \sum_{i=1}^{2m_1} \sum_{j,k=1}^{2m_2} ((\phi e_i(\ln f))^2 \delta_{jk} + (e_i(\ln f) g(v_j, Tv_k))^2  \nonumber  \\
& &-2\phi e_i(\ln f) \delta_{jk}\cdot e_i(\ln f) g(v_j, Tv_k)) + 4m_2 \sin^2 \theta,    \nonumber
\end{eqnarray}
where $\delta_{jk}$ is the Kronecker delta for $1\leq j,k \leq 2m_2$.

But
\begin{eqnarray}
\sum_{i=1}^{2m_1} (\phi e_i(\ln f))^2  &=& \sum_{i=1}^{2m_1} g(\phi e_i, \nabla(\ln f))^2             \label{eq: ineq03} \\
&=& \sum_{i=1}^{2m_1} g(e_i, \phi \nabla(\ln f))^2  \nonumber  \\
&=& g(\phi \nabla(\ln f), \phi \nabla(\ln f))     \nonumber  \\
&=& ||\phi \nabla(\ln f)||^2,     \nonumber
\end{eqnarray}
\begin{eqnarray}
\sum_{i=1}^{2m_1} (e_i(\ln f))^2  &=& \sum_{i=1}^{2m_1} g(e_i, \nabla(\ln f))^2             \label{eq: ineq04} \\
&=& g(\nabla(\ln f), \nabla(\ln f)) - (\eta(\nabla(\ln f)))^2  \nonumber  \\
&=& g(\phi \nabla(\ln f), \phi \nabla(\ln f))     \nonumber  \\
&=& ||\phi \nabla(\ln f)||^2,     \nonumber
\end{eqnarray}
\begin{equation}\label{eq: ineq05}
\delta_{jk} g(v_j, Tv_k) = 0,
\end{equation}
By Remark \ref {rem: sslant02},
\begin{eqnarray}
& &\sum_{j,k=1}^{2m_2} g(v_j, Tv_k)^2            \label{eq: ineq06} \\
&=& \sum_{k=1}^{m_2} \sum_{j=1}^{2m_2} g(v_j, Tv_k)^2 + \sum_{k=1}^{m_2} \sum_{j=1}^{2m_2} g(v_j, Tv_{m_2+k})^2  \nonumber  \\
&=& \sum_{k=1}^{m_2} g(\sec \theta Tv_k, Tv_k)^2 + \sum_{k=1}^{m_2} g(v_k, \sec \theta (-\cos^2 \theta)v_k)^2    \nonumber  \\
&=& \sum_{k=1}^{m_2} \sec^2 \theta \cdot \cos^4 \theta + \sum_{k=1}^{m_2} \cos^2 \theta    \nonumber  \\
&=& 2m_2 \cos^2 \theta.     \nonumber
\end{eqnarray}
Applying (\ref {eq: ineq03}), (\ref {eq: ineq04}), (\ref {eq: ineq05}), (\ref {eq: ineq06}) to (\ref {eq: ineq02}), we have
\begin{align*}
||h||^2
&= \sum_{i,j,k=1}^{2m_2} g(h(v_i, v_j), w_k)^2 + 2\csc^2 \theta (2m_2 ||\phi \nabla(\ln f)||^2   \\
&\ + 2m_2 \cos^2 \theta ||\phi \nabla(\ln f)||^2) + 4m_2 \sin^2 \theta
\end{align*}
so that
$$
||h||^2 \geq 4m_2 (\csc^2 \theta + \cot^2 \theta) ||\phi \nabla(\ln f)||^2 + 4m_2 \sin^2 \theta
$$
with equality holding if and only if $g(h(v_i, v_j), w_k) = 0$ for $1\leq i,j,k \leq 2m_2$.

Therefore, the result follows.
\end{proof}

In the same way, we get

\begin{theorem}
Let $M = B\times_f \overline{F}$ be a $m$-dimensional nontrivial warped product proper pointwise semi-slant submanifold of
a $(2n+1)$-dimensional cosymplectic manifold $(N,\phi,\xi,\eta,g)$ with the semi-slant function $\theta$
such that $\mathcal{D}_1 = TB$, $\mathcal{D}_2 = T\overline{F}$, and $\xi$ is tangent to $M$.

Assume that $n = m_1 + 2m_2$.

Then we have
\begin{equation}\label{eq: ineq07}
||h||^2 \geq 4m_2 (\csc^2 \theta + \cot^2 \theta) ||\phi \nabla(\ln f)||^2
\end{equation}
with equality holding if and only if $g(h(Z, W), V) = 0$ for $Z,W\in \Gamma(T\overline{F})$ and $V\in \Gamma(TM^{\perp})$.
\end{theorem}

\begin{theorem}
Let $M = B\times_f \overline{F}$ be a $m$-dimensional nontrivial warped product proper pointwise semi-slant submanifold of
a $(2n+1)$-dimensional Kenmotsu manifold $(N,\phi,\xi,\eta,g)$ with the semi-slant function $\theta$
such that $\mathcal{D}_1 = TB$, $\mathcal{D}_2 = T\overline{F}$, and $\xi$ is tangent to $M$.

Assume that $n = m_1 + 2m_2$.

Then we have
\begin{equation}\label{eq: ineq08}
||h||^2 \geq 4m_2 (\csc^2 \theta + \cot^2 \theta) ||\phi \nabla(\ln f)||^2
\end{equation}
with equality holding if and only if $g(h(Z, W), V) = 0$ for $Z,W\in \Gamma(T\overline{F})$ and $V\in \Gamma(TM^{\perp})$.
\end{theorem}

Let $M = B\times_f \overline{F}$ be a $m$-dimensional nontrivial warped product proper pointwise semi-slant submanifold of
a $(2n+1)$-dimensional almost contact metric manifold $(N,\phi,\xi,\eta,g)$ with the semi-slant function $\theta$
such that $\mathcal{D}_1 = TB$, $\mathcal{D}_2 = T\overline{F}$, and $\xi$ is normal to $M$ with $\xi\in \Gamma(\mu)$.

Then by Propositin \ref {prop: sslant01}, $\mu$ is $\phi$-invariant.

Using Remark \ref {rem: sslant02}, we can choose a local orthonormal frame
$\{ e_1,e_2,\cdots,$ $e_{2m_1},$ $v_1,\cdots,v_{2m_2},w_1,\cdots,w_{2m_2},u_1,\cdots,u_{2r+1} \}$ of $TN$ such that
$\{ e_1,\cdots,e_{2m_1} \} \subset \Gamma(\mathcal{D}_1)$, $\{ v_1,\cdots,v_{2m_2} \} \subset \Gamma(\mathcal{D}_2)$,
$\{ w_1,\cdots,w_{2m_2} \} \subset \Gamma(F\mathcal{D}_2)$, $\{ u_1,\cdots,u_{2r+1} \} \subset \Gamma(\mu)$ with the following conditions:
\begin{enumerate}
\item $e_{m_1+i} = \phi e_i$, $1\leq i \leq m_1$,

\item $v_{m_2+i} = \sec \theta Tv_i$, $1\leq i \leq m_2$,

\item $w_{i} = \csc \theta Fv_i$, $1\leq i \leq 2m_2$,

\item $u_{r+i} = \phi u_i$, $1\leq i \leq r$, $u_{2r+1} = \xi$.
\end{enumerate}
So, we have $m = 2m_1 + 2m_2$ and $n = m_1 + 2m_2 + r$.

Notice that if $N$ is Sasakian, then from Theorem \ref {thm: ang02},  there does not exist such a proper pointwise semi-slant submanifold $M$ of $N$.

Using these notations, in a similar way, we obtain

\begin{theorem}
Let $M = B\times_f \overline{F}$ be a $m$-dimensional nontrivial warped product proper pointwise semi-slant submanifold of
a $(2n+1)$-dimensional Kenmotsu manifold $(N,\phi,\xi,\eta,g)$ with the semi-slant function $\theta$
such that $\mathcal{D}_1 = TB$, $\mathcal{D}_2 = T\overline{F}$, and $\xi$ is normal to $M$ with $\xi\in \Gamma(\mu)$.

Assume that $n = m_1 + 2m_2$.

Then we have
\begin{equation}\label{eq: ineq09}
||h||^2 \geq 4m_2 (\csc^2 \theta + \cot^2 \theta) ||\nabla(\ln f)||^2 + 2m_1
\end{equation}
with equality holding if and only if $g(h(Z, W), V) = 0$ for $Z,W\in \Gamma(T\overline{F})$ and $V\in \Gamma(TM^{\perp})$.
\end{theorem}

\begin{proof}
Since $\mu = < \xi >$, we obtain
\begin{align*}
||h||^2
&= \sum_{i,j=1}^{2m_1} g(h(e_i, e_j), h(e_i, e_j)) + \sum_{i,j=1}^{2m_2} g(h(v_i, v_j), h(v_i, v_j))   \\
&\ + 2 \sum_{i=1}^{2m_1} \sum_{j=1}^{2m_2} g(h(e_i, v_j), h(e_i, v_j))  \\
&= \sum_{i,j=1}^{2m_1}  (\sum_{k=1}^{2m_2} g(h(e_i, e_j), w_k)^2 + (\eta(h(e_i, e_j)))^2) \\
&\ + \sum_{i,j=1}^{2m_2} (\sum_{k=1}^{2m_2} g(h(v_i, v_j), w_k)^2 + (\eta(h(v_i, v_j)))^2)  \\
&\ + 2 \sum_{i=1}^{2m_1} (\sum_{j,k=1}^{2m_2} g(h(e_i, v_j), w_k)^2 + (\eta(h(e_i, v_j)))^2).
\end{align*}
Using (\ref {eq: struc13}),  we can easily check that $\eta(h(e_i, e_j)) = -\delta_{ij}$ and $\eta(h(e_i, v_k)) = 0$
for $1\leq i,j \leq 2m_1$ and $1\leq k \leq 2m_2$ so that by using Lemma \ref {lem: warp03},
\begin{align*}
||h||^2
&= 2m_1 + \sum_{i,j=1}^{2m_2} (\sum_{k=1}^{2m_2} g(h(v_i, v_j), w_k)^2 + (\eta(h(v_i, v_j)))^2)  \\
&\ + 2\csc^2 \theta \sum_{i=1}^{2m_1} \sum_{j,k=1}^{2m_2} (-\eta(e_i) \eta(v_j) \eta(Fv_k)      \\
& - \phi e_i(\ln f) g(v_j, v_k) + (e_i - \eta(e_i)\xi)(\ln f) g(v_j, Tv_k))^2 \\
&= 2m_1 + \sum_{i,j=1}^{2m_2} (\sum_{k=1}^{2m_2} g(h(v_i, v_j), w_k)^2 + (\eta(h(v_i, v_j)))^2)  \\
&\ + 2\csc^2 \theta \sum_{i=1}^{2m_1} \sum_{j,k=1}^{2m_2} ((\phi e_i(\ln f))^2 \delta_{jk} + (e_i(\ln f) g(v_j, Tv_k))^2     \\
& -2 \phi e_i(\ln f) \delta_{jk} \cdot e_i(\ln f) g(v_j, Tv_k)).
\end{align*}
In a similar way to the proof of Theorem \ref {thm: ineq01}, we also derive the following:
\begin{align*}
&\sum_{i=1}^{2m_1} (\phi e_i(\ln f))^2 = ||\nabla(\ln f)||^2,   \\
&\sum_{i=1}^{2m_1} (e_i(\ln f))^2 = ||\nabla(\ln f)||^2,     \\
&\sum_{j,k=1}^{2m_2}  (g(v_j, Tv_k))^2 = 2m_2 \cos^2 \theta, \\
&\delta_{jk} g(v_j, Tv_k) = 0
\end{align*}
so that
$$
||h||^2 \geq 2m_1 +  4m_2 (\csc^2 \theta + \cot^2 \theta) ||\nabla(\ln f)||^2
$$
with equality holding if and only if $g(h(v_i, v_j), w_k) = 0$ and $g(h(v_i, v_j), \xi) = 0$ for $1\leq i,j,k \leq 2m_2$.

Therefore, the result follows.
\end{proof}

In the same way, we get

\begin{theorem}
Let $M = B\times_f \overline{F}$ be a $m$-dimensional nontrivial warped product proper pointwise semi-slant submanifold of
a $(2n+1)$-dimensional cosymplectic manifold $(N,\phi,\xi,\eta,g)$ with the semi-slant function $\theta$
such that $\mathcal{D}_1 = TB$, $\mathcal{D}_2 = T\overline{F}$, and $\xi$ is normal to $M$ with $\xi\in \Gamma(\mu)$.

Assume that $n = m_1 + 2m_2$.

Then we have
\begin{equation}\label{eq: ineq010}
||h||^2 \geq 4m_2 (\csc^2 \theta + \cot^2 \theta) ||\nabla(\ln f)||^2
\end{equation}
with equality holding if and only if $g(h(Z, W), V) = 0$ for $Z,W\in \Gamma(T\overline{F})$ and $V\in \Gamma(TM^{\perp})$.
\end{theorem}

\section{Examples}\label{exam2}

\begin{example}
Define a map $i : \mathbb{R}^4 \mapsto \mathbb{R}^{11}$ by
\begin{align*}
&i(x_1,x_2,x_3,x_4) = (y_1,y_2,\cdots,y_{10},t) = (x_2\sin x_3,x_1\sin x_3,    \\
&x_2\sin x_4,x_1\sin x_4,x_2\cos x_3,x_1\cos x_3,x_2\cos x_4,x_1\cos x_4,x_3,x_4,0)
\end{align*}
Let $M := \{ (x_1,x_2,x_3,x_4)\in \mathbb{R}^4 \mid 0< x_1, x_2 < 1, \ 0< x_3, x_4 < \frac{\pi}{2} \}$.

We define $(\phi,\xi,\eta,g)$ on $\mathbb{R}^{11}$ as follows:
\begin{align*}
&\phi(a_1\tfrac{\partial}{\partial y_1} + \cdots + a_{10}\tfrac{\partial}{\partial y_{10}} + a_{11}\tfrac{\partial}{\partial t}) :=
\sum_{i=1}^5 (-a_{2i}\tfrac{\partial}{\partial y_{2i-1}} + a_{2i-1}\tfrac{\partial}{\partial y_{2i}}),     \\
&\xi := \tfrac{\partial}{\partial t}, \ \eta := dt, \ a_i\in \mathbb{R}, \ 1\leq i \leq 11,
\end{align*}
$g$ is the Euclidean metric on $\mathbb{R}^{11}$.

We easily check that $(\phi,\xi,\eta,g)$ is an almost contact metric structure on $\mathbb{R}^{11}$.
Then $M$ is a pointwise semi-slant submanifold of $\mathbb{R}^{11}$ with the semi-slant function
$\displaystyle{k(x_1,x_2,x_3,x_4) = \arccos (\frac{1}{x_1^2 + x_2^2 + 1})}$
such that $\xi$ is normal to $M$ and
\begin{align*}
\mathcal{D}_1 = &< \sin x_3 \tfrac{\partial}{\partial y_2} + \cos x_3 \tfrac{\partial}{\partial y_6}
+ \sin x_4 \tfrac{\partial}{\partial y_8} + \cos x_4 \tfrac{\partial}{\partial y_{10}},      \\
&\sin x_3 \tfrac{\partial}{\partial y_1} + \cos x_3 \tfrac{\partial}{\partial y_5}
+ \sin x_4 \tfrac{\partial}{\partial y_7} + \cos x_4 \tfrac{\partial}{\partial y_{9}} >,
\end{align*}
\begin{align*}
\mathcal{D}_2 = &< x_2\cos x_3 \tfrac{\partial}{\partial y_1} + x_1\cos x_3 \tfrac{\partial}{\partial y_2} + \tfrac{\partial}{\partial y_3}
-x_2 \sin x_3 \tfrac{\partial}{\partial y_5} -x_1 \sin x_3 \tfrac{\partial}{\partial y_{6}},      \\
&\tfrac{\partial}{\partial y_4} + x_2\cos x_4 \tfrac{\partial}{\partial y_7} + x_1\cos x_4 \tfrac{\partial}{\partial y_8}
-x_2 \sin x_4 \tfrac{\partial}{\partial y_9} -x_1 \sin x_4 \tfrac{\partial}{\partial y_{10}} >.
\end{align*}
Notice that $(\mathbb{R}^{11},\phi,\xi,\eta,g)$ is cosymplectic.
\end{example}

\begin{example}
Define a map $i : \mathbb{R}^5 \mapsto \mathbb{R}^{7}$ by
\begin{align*}
&i(x_1,x_2,\cdots,x_5) = (y_1,y_2,\cdots,y_{6},t)    \\
&= (x_3,x_1, x_5,\sin x_4,0,\cos x_4,x_2).
\end{align*}
Let $M := \{ (x_1,x_2,\cdots,x_5)\in \mathbb{R}^5 \mid 0< x_4 < \frac{\pi}{2} \}$.

We define $(\phi,\xi,\eta,g)$ on $\mathbb{R}^{7}$ as follows:
\begin{align*}
&\phi(a_1\tfrac{\partial}{\partial y_1} + \cdots + a_{6}\tfrac{\partial}{\partial y_{6}} + a_{7}\tfrac{\partial}{\partial t}) :=
\sum_{i=1}^3 (-a_{2i}\tfrac{\partial}{\partial y_{2i-1}} + a_{2i-1}\tfrac{\partial}{\partial y_{2i}}),     \\
&\xi := \tfrac{\partial}{\partial t}, \ \eta := dt, \ a_i\in \mathbb{R}, \ 1\leq i \leq 7,
\end{align*}
$g$ is the Euclidean metric on $\mathbb{R}^{7}$.
It is easy to check that $(\phi,\xi,\eta,g)$ is an almost contact metric structure on $\mathbb{R}^{7}$.

Then $M$ is a pointwise semi-slant submanifold of $\mathbb{R}^{7}$ with the semi-slant function $k(x_1,\cdots,x_5) = x_4$
such that $\xi$ is tangent to $M$ and
\begin{align*}
&\mathcal{D}_1 = < \tfrac{\partial}{\partial y_1}, \tfrac{\partial}{\partial y_2}, \xi >   \\
&\mathcal{D}_2 = < \tfrac{\partial}{\partial y_3}, \cos x_4 \tfrac{\partial}{\partial y_4}
- \sin x_4 \tfrac{\partial}{\partial y_6} >.
\end{align*}
\end{example}

\begin{example}
Let $(N,\phi,\xi,\eta,g_N)$ be an almost contact metric manifold.
Let $M$ be a submanifold of a hyperk\"{a}hler manifold $(\overline{M},J_1,J_2,J_3,g_{\overline{M}})$ such that $M$ is complex
with respect to the complex structure $J_1$ (i.e., $J_1(TM) = TM$) and totally real with respect to the complex structure
$J_2$ (i.e., $J_2(TM) \subset TM^{\perp}$) \cite {B}.
Let $f : \overline{M} \mapsto [0, \frac{\pi}{2}]$ be a $C^{\infty}$-function.
Let $\overline{N} := \overline{M} \times N$ with the natural projections $\pi_1 : \overline{N} \mapsto \overline{M}$ and $\pi_2 : \overline{N} \mapsto N$.

We define $(\overline{\phi},\overline{\xi},\overline{\eta},\overline{g})$ on $\overline{N}$ as follows:
\begin{align*}
&\overline{\phi}(X + Y) := \cos (f\circ \pi_1) J_1 X - \sin (f\circ \pi_1) J_2 X + \phi Y,     \\
&\overline{\xi} := \xi, \quad \overline{\eta} := \eta,      \\
&\overline{g}(Z, W) := g_{\overline{M}}(d\pi_1(Z), d\pi_1(W)) + g_N(d\pi_2(Z), d\pi_2(W))
\end{align*}
for $X\in \Gamma(T\overline{M})$, $Y\in \Gamma(TN)$, $Z,W\in \Gamma(T\overline{N})$.

Here, $\overline{\xi}$ is exactly the horizontal lift of $\xi$ along $\pi_2$ and $\overline{\eta}(Z) := \eta(d\pi_2(Z))$.
Conveniently, we identify  a vector field on $\overline{M}$ (or on $N$) with its horizontal lift.

We can easily check that $(\overline{\phi},\overline{\xi},\overline{\eta},\overline{g})$ is an almost contact metric structure on $\overline{N}$.

Then $M\times N$ is a pointwise semi-slant submanifold of an almost contact metric manifold
$(\overline{N},\overline{\phi},\overline{\xi},\overline{\eta},\overline{g})$  with the semi-slant function $f\circ \pi_1$
such that $\overline{\xi}$ is tangent to $M\times N$ and $\mathcal{D}_1 = TN$, $\mathcal{D}_2 = TM$.
\end{example}

\end{document}